\documentclass[a4paper,11pt,reqno]{amsart}

\usepackage{amsmath,amsthm,latexsym,amscd,amssymb}
\usepackage[all]{xy}
\usepackage{colortbl}
\usepackage{hyperref}
\usepackage{fancyhdr}
\usepackage{amscd}
\usepackage{url}
\numberwithin{equation}{section}

\begin{document}
%\swapnumbers
\newtheorem{prop}{Proposition}[section]
\newtheorem{lem}[prop]{Lemma}
\newtheorem{theorem}[prop]{Theorem}
\newtheorem{cor}[prop]{Corollary}
\newtheorem{conj}[prop]{Conjecture}
\newtheorem{que}[prop]{Question}
\theoremstyle{definition}
\newtheorem{dfn}[prop]{Definition}
\newtheorem{ex}[prop]{Example}
\theoremstyle{remark}
\newtheorem{rem}[prop]{Remark}

\def\no{\noindent}
% quando il footnote e' della stessa dimensione:
\newcommand{\foot}[1]{\footnote{\begin{normalsize}#1\end{normalsize}}}

% o p e r a t o r i     m a t e m a t i c i    

\def\dim{\mathop{\rm dim}}\def\codim{\mathop{\rm codim}}\def\Re{\mathop{\rm Re}}
\def\Im{\mathop{\rm Im}}\def\I{\mathop{\rm I}}
\def\Id{\mathop{\rm Id}}\def\grad{\mathop{\rm grad}}
\def\vol{\mathop{\rm vol}}\def\SU{\mathop{\rm SU}}
\def\SO{\mathop{\rm SO}}\def\Aut{\mathop{\rm Aut}}
\def\End{\mathop{\rm End}\nolimits}\def\GL{\mathop{\rm GL}}\def\Cinf{\mathop{\mathcal C^{\infty}}}\def\Ker{\mathop{\rm Ker}}
\def\Coker{\mathop{\rm Coker}}\def\dom{\mathop{\rm Dom}}
\def\Hom{\mathop{\rm Hom}}\def\Ch{\mathop{\rm Ch}}
\def\sign{\mathop{\rm sign}}\def\loc{\mathop{\rm loc}}
\def\AS{\mathop{\rm AS}}\def\spec{\mathop{\rm spec}}
\def\Ric{\mathop{\rm Ric}}\def\ch{\mathop{\rm ch}}
\def\ch{\mathop{\rm ch}\nolimits}\def\rk{\mathop{\rm rk}}
\def\ev{\mathop{\rm ev}}\def\id{\mathop{\rm id}}
\def\Cli{\mathbb{C}l(1)}\def\w{\wedge}
\def\Diffeo{\mathop{\rm Diffeo}}
\def\curv{\mathop{\rm curv}}

%                  l e t t e r e     g r e c h e 
\def\Fi{\Phi}\def\phi{\varphi}
\def\de{\delta}\def\e{\eta}\def\ep{\varepsilon}\def\ro{\rho}
\def\a{\alpha}\def\o{\omega}
\def\O{\Omega}\def\b{\beta}\def\la{\lambda}
\def\th{\theta}\def\s{\sigma}\def\t{\tau}
\def\g{\gamma}\def\D{\Delta}\def\G{\Gamma}

\def\R{\mathbin{\mathbb R}}
\def\Rn{\R^{n}}\def\C{\mathbb{C}}
\def\Cm{\mathbb{C}^{m}}\def\Cn{\mathbb{C}^{n}}

% Calligraphic and bold abbreviations
\def\cA{\mathcal A}\def\cL{{\mathcal L}}
\def\cO{{\mathcal O}}\def\cT{{\mathcal T}}\def\cU{{\mathcal U}}\def\cI{{\mathcal I}}
\def\cD{{\mathcal D}}\def\cF{\mathcal F}\def\cP{{\mathcal P}}\def\cH{{\mathcal H}}\def\cL{{\mathcal L}}\def\cB{{\mathcal B}}\def\cG{\mathcal G}
\def\cW{\mathcal W}\def\cR{\mathcal R}

% N O R M A   D I    U N    V E T T O R E 
\newcommand{\n}[1]{\left\| #1\right\|}% grande palla B di raggio R
\newcommand{\op}{\mathrm{op}}

%%%%%%%%%%%%%%%%%%%%%%%%%%%%%%%%%%%%%%%%%%%%%%%%%%%%%%%%%%%%%%%%%%%%%%%%%%%%%%%%
%%%%%%%%%%%%%%%%%%%%%%%%%%%%%%%%%%%%%%%%%%%%%%%%%%%%%%%%%%%%%%%%%%%%%%%%%%%%%%%5

% TILDE
\def\Mt{\tilde{M}}
\def\Et{\tilde{E}}
\def\Vt{\tilde{V}}
\def\Xt{\tilde{X}}
\def\N{\mathcal{V}}
\def\Cd{\mathbb{C}^{d}}

%%%%%%        T R A C C E   %%%%%%%%%%%%%%%
\def\tr{\mathop{\rm tr}\nolimits}
\def\trt{\mathop{\rm tr{}_{\tau}}}
\def\StrA{\mathop{ \rm Str_{\mathcal{A}}}\nolimits}
\def\Strt{\mathop{ \rm Str_{\t}}}
\def\SpA{\mathop{ \rm Str^\pm_{\mathcal{A}}}}
\def\TR{\mathop{\rm TR}}\def\trace{\mathop{\rm trace}}
\def\STR{\mathop{\rm STR}}\def\trG{\mathop{\rm tr_\Gamma}}
\def\TRG{\mathop{\rm TR_\Gamma}}\def\Tr{\mathop{\rm Tr}}
\def\Str{\mathop{\rm Str}\nolimits}

\def\Cl{\mathop{\rm Cl}}
\def\Op{\mathop{\rm Op}}\def\supp{\mathop{\rm supp}}
\def\scal{\mathop{\rm scal}}\def\ind{\mathop{\rm ind}}
\def\dim{\mathop{\rm dim}}\def\Ind{\mathop{\rm Ind}\,}
\def\Diff{\mathop{\rm Diff}}\def\T{\mathcal{T}}
\def\Cliff{\mathop*{\rm Cliff}\,}
\def\SL{\mathop{\rm SL}\nolimits}

%D I R A C :
\newcommand\Di{D\kern-7pt/}

% INDEX CLASS
\def\aI{\mathop{\rm Ind_a}}
\def\ds{\displaystyle}

% Evidenziatore 
\definecolor{light}{gray}{.95}
\newcommand{\highlight}[1]{
$\phantom .$
\bigskip
\par\noindent
\colorbox{light}{\begin{minipage}{13.5 cm}#1\end{minipage}}
\bigskip
\par\noindent
}

%%%   COMMENTI, DOMANDE

%\newenvironment{name}[num]{before}{after}
\newenvironment{com}%
{\par \vspace{0cm}\ \\ \noindent 
\%\color{blue}\% }{\%\color{black}\%\vspace{0cm}}

\newcommand{\question}[1]{\color{blue}\vspace{5 mm}\par \noindent
\marginpar{\color{blue}\textsc{Q }} \framebox{\begin{minipage}[c]{0.95
\textwidth} \raggedright \tt #1 \end{minipage}}\vspace{5 mm}\color{black}
\par}
\newcommand{\Komm}[1]{\color{blue}\vspace{2 mm}\par \noindent 
\marginpar{\color{blue} Comment} \framebox{\begin{minipage}[c]{0.95
\textwidth} \raggedright \tt #1 \end{minipage}}\vspace{2 mm}\color{black}\par}

\newenvironment{select}[1]
{\color{blue} \marginpar{#1 $\top$} 
\color{blue} }{
\marginpar{\color{blue}\hspace{.7cm} $\bot$}
\color{black}}

%  CHECKMARK in rosso
\newcommand\ok{\color{red}\checkmark\color{black}}
%%%

%%%%%%%%%%%%%%%%%%%%%%%%%%%%%%%%%%%%%%%%%%%%%%%%%%%%%%%%%%%%%%%%%%%%%%%%%%%%%%
%%%%%%%%%%%%%%%%%%%%%%%%%%%%%%%%%%%
%         E n d    o f    T o p m a t t e r   %%%%%%%%%

\title[Large time limit and local $L^2$-index theorems for families]{Large time limit and\\local $L^2$-index theorems for families}

%\author{Sara Azzali, Sebastian Goette, Thomas Schick}
\author{Sara Azzali}
\address{Sara Azzali, Institut f\"ur Mathematik, Universit\"at Potsdam \\ Am Neuen Palais 10, 14469 Potsdam, Germany}
\email{azzali@uni-potsdam.de}
\thanks{}
\author{Sebastian Goette}
\address{Sebastian Goette\\ Mathematisches Institut\\
Universit\"at Freiburg\\
Eckerstr.~1\\
79104 Freiburg\\
Germany}
\email{sebastian.goette@math.uni-freiburg.de}\thanks{}
\author{Thomas Schick}
\address{Thomas Schick\\ Mathematisches Institut\\
Georg August Universit\"at G\"ottingen\\
Bunsenstr.~3-5 \\37073 G\"ottingen\\
Germany}
\email{thomas.schick@math.uni-goettingen.de}
\thanks{}

\begin{abstract}
We compute explicitly, and without any extra regularity assumptions, the large
time limit of the fibrewise heat operator for Bismut--Lott type
superconnections in the $L^2$-setting.
This is motivated by index theory on certain non-compact spaces (families of manifolds with cocompact group action) where the convergence of the heat operator at large time implies refined $L^2$-index formulas. 

As applications, we prove a local $L^2$-index theorem for families of
signature operators and an $L^2$-Bismut--Lott theorem, expressing the Becker--Gottlieb
transfer of flat bundles in terms of Kamber--Tondeur classes. With slightly
stronger regularity we obtain the respective refined versions: we construct
$L^2$-eta forms and  $L^2$-torsion forms as transgression forms.
\end{abstract}

\maketitle
\tableofcontents

\section{Introduction} \label{intro}
The origin of local index theory is the fundamental observation by Atiyah--Bott and McKean--Singer that the index of an elliptic differential operator is expressed in terms of the trace of its heat operator. 
For Dirac type operators, the heat-kernel's supertrace provides an interpolation between the local geometry, at small time limit, and the index, at large time.
 
The superconnection formalism, developed by Quillen and Bismut, is the essential tool to apply the heat-kernel approach to the analytic theory of families of operators, and naturally furnishes transgression formulas \cite{Q, Bi, BGV}.
Bismut's heat-kernel proof of the Atiyah--Singer family index theorem provides then a fundamental refinement of the cohomological index formula to the level of differential forms: when the kernels form a bundle, the heat operator converges, \emph{as $t\to\infty$}, to an explicit  differential form obtained as the large time limit of the \emph{superconnection Chern character}. The small time limit is the index density, and the local index formula is completed by the transgression term, involving a secondary invariant, the \emph{eta form} of the family \cite{Bi, BC, BF}. 

A parallel result is the Bismut--Lott index theorem for flat vector bundles, proving a Riemann--Roch--Grothendieck theorem for the direct images along a submersion of flat bundles \cite{BL}: it says that the Kamber--Tondeur classes of the fibrewise cohomology twisted by a globally flat vector bundle $F$ are equal to the Becker--Gottlieb transfer of the classes of $F$. The transgression term of the refinement at differential forms level in Bismut--Lott's formula involves the secondary invariant \emph{higher analytic torsion} \cite{BL,Lo}. 

Local index theory was extended to non-commutative families in many different contexts and with a variety of approaches: for example  by Lott to the higher index theory for coverings \cite{Lo3}; by Heitsch--Lazarov and Benameur--Heitsch for foliations with Hausdorff graph in the  Haefliger cohomology context \cite{HL, BH2}; by Gorokhosky--Lott for \'etale groupoids \cite{GL}. 

While the small time limit is in every case the local index density,  it is no longer true that 
the large time limit of the heat operator always gives the index class, as soon as the fibres (or leaves) are non-compact. First of all, the heat operator is in general not convergent as $t\to \infty$; moreover, the index class is in general different from the so called \emph{index bundle},  defined, when the projection onto the fibrewise kernel is transversally smooth, as the Chern character of the corresponding $K$-theory class. 

Heitsch--Lazarov and Benameur--Heitsch investigated this problem for longitudinal Dirac type operators on a foliation, employing the superconnection formalism on Haefliger forms. Their index theorem in Haefliger cohomology applies to longitudinal operators admitting an index bundle, and under the further assumptions that the $(0,\ep)$-spectral projection is transversally smooth and the leafwise Novikov--Shubin invariants are greater than half of the foliation's codimension \cite{HL,BH2}. This regularity ensures that the heat operator converges as $t\to\infty$ to the index bundle, and in particular proves the equality of the Chern character of index bundle and index class \cite{HL, BH2}. 
Benameur, Heitsch and Wahl recently showed an example where this equality does not hold, for a family of Dirac operators whose Novikov--Shubin invariants are just off the regularity condition \cite{BHW}.

The approach of \cite{HL, BH2} to the large time limit, inspired by the work of Gong--Rothenberg on families of coverings \cite{GR}, makes use of the decomposition of the spectrum of the Dirac Laplacian into $\{0\}\cup (0,\ep)\cup [\ep,\infty)$ and hence requires the assumption of smooth $(0,\ep)$-spectral projections, along with lower bounds on the Novikov--Shubin invariants.

To ask for transversal smoothness of the $(0,\ep)$-spectral projection is
a very strong condition, and difficult to be verified in the case of the
geometrically most relevant operators.

On the other hand, if we focus on the Laplacian, which is the square of the Euler  and signature operators, it is known that it has 
some intrinsic regularity, coming from the topological nature of its kernel.
Even in the non-compact settings of coverings and measured foliations, one can usually translate this into a nice behaviour of the large time limit of its heat operator:  for instance this is exploited by Cheeger and Gromov in the proof of the metric independence of the $L^2$-rho invariant of the signature operator \cite[(4.12)]{CG2}.
This fact suggests that for a family of longitudinal Laplacians it should be possible to prove the convergence of the heat operator at the large time limit without assuming extra regularity conditions.

Motivated by this idea, in this paper we investigate the large time limit of the heat operator for families of Euler and signature operators in the $L^2$-setting of \emph{families of normal coverings}. We consider this as a first step to understand more general foliated manifolds. 

Given a smooth fibre bundle $p\colon E\to B$ with compact fibre $Z$, a family of normal coverings $(\widetilde E, \G)\to B$ consists of a bundle of discrete groups $\G\to B$ and of a covering $\pi\colon\widetilde E\to E$ such that ~$\pi_b\colon\widetilde E_b\to E_b$ is a normal covering
with group of covering transformations~$\Gamma_b$ for all~$b\in B$. The family under consideration is then the Euler operator $d^Z+d^{Z,*}$, where $d^Z$ is the fibrewise de Rham operator twisted by a globally flat bundle $\mathcal F\to E$ of $\mathcal A$-Hilbert modules, for a given finite von Neumann algebra $\mathcal A$. 

Our first result, Theorem \ref{HeatThm2}, is the explicit computation of the limit as $t\to \infty$ of the heat operator for  the Bismut--Lott superconnection, paired with the finite trace $\t$ on $\mathcal A$, without assuming any regularity hypothesis on the spectrum.   As applications, in Section \ref{secL2indthm} we prove the $L^2$-Bismut--Lott index theorem for flat bundles and the $L^2$-index theorem for the family of signature operators. In particular we show that, for the signature operator in the $L^2$-setting, the Chern character of index class and index bundle have the same pairing with the trace $\t$: hence our equality points out the special behavior of the signature operator with respect to the general question by Benameur--Heitsch of the comparison of these two objects characters on a foliation \cite{BH2, BHW}.

We realize the computation of the large time limit using two main ingredients. The first one is a fundamental observation due to Bismut and Lott \cite{BL, Lo2}: the superconnection adapted to the longitudinal signature operator is given by $\mathbb A=\frac{1}{2}(d^E+d^{E,*})$,  where $d^{E,*}$ is the adjoint superconnection of $d^E$. Since $d^E$ is flat, the curvature $\mathbb A^2$ has another \lq\lq square root", the operator $\mathbb{X}=\frac{1}{2}(d^{E,*}-d^{E})$ which satisfies $\mathbb A^2=-\mathbb X^2$. The operator $\mathbb X$  does not involve transversal derivatives because it is the difference of two superconnections, and we exploit this property very carefully in the Duhamel expansion of the heat operator $e^{\mathbb{X}_t^2}$.
The second ingredient is a new method of estimating the terms of the perturbative expansion of $e^{\mathbb{X}_t^2}$, developed in Section \ref{heat-kernel}.

Our technique only applies to the case of families of Euler and signature operators twisted by globally flat bundles of $\cA$-modules, because we use deeply the fact that $d^E$ is flat, and the existence of the operator $\mathbb X$. On the other hand, we believe that our estimates can be applied almost immediately to foliations, at least taking the point of view of Haefliger forms, where the trace is defined by a local push forward, using the local structure of fibration \cite{HL1, HL, BH2}.

The next result, in Section \ref{refined}, is the construction of the $L^2$-eta and $L^2$-torsion forms as transgression forms. To this aim, we implement the estimates on the Duhamel expansion, and prove that
these $L^2$-eta and $L^2$-torsion are well defined if the fibre is of
determinant class and $L^2$-acyclic, or if the Novikov--Shubin invariants are
positive. Under these assumptions we even  prove differential form refinements
of the $L^2$-index theorems.  Compared to the construction of the
$L^2$-torsion form  by Gong and Rothenberg \cite{GR}, our approach does not
need the smoothness of the spectral projection $\chi_{(0,\ep)}(D)$, and holds
for families of manifolds of determinant class (provided they are
$L^2$-acyclic). This is indeed an improvement, as recently Grabowski proved that there exist closed manifolds with Novikov--Shubin
invariant equal to zero \cite{Gra} (which are of determinant class by \cite{Sc2}).

In the last Section, we investigate the properties of the $L^2$-rho form of the signature operator.

\section{Setup}

In this section we describe the different situations we consider 
as $L^2$-settings for geometric families and which we treat in the paper.
In \ref{Sect.NCF}, we consider normal coverings of a fibre bundle, and therefore we work on families with coefficients in a flat bundle of finitely generated Hilbert $\mathcal A$-modules, where $\cA$ is a finite von Neumann algebra.
In \ref{Sect.FNC} we generalize to families of normal coverings.

\subsection{Normal coverings of fibre bundles}\label{Sect.NCF}
Let $\tilde p\colon\widetilde E\rightarrow B$ be a 
%Riemannian submersion,
smooth fibre bundle,
and let $\G$ act fibrewise freely and properly discontinuously on $\tilde E$
such that the fibres of $p\colon E=\widetilde E/\G\rightarrow B$ are compact.
%In particular, $\G$ is a quotient of the fundamental group of the fibres. (% only if the fibres are connected!)
Let~$\pi$ denote the quotient map~$\widetilde E\to E$. We call this setting a \emph{normal covering of the fibre bundle $p\colon E\rightarrow B $}.
%Let $Z \to E\stackrel{p}{\to} B$ be a smooth fibre bundle with connected closed fibers $Z_b=\pi^{-1}(b)$ of dimension $n$.  Let $\pi\colon \widetilde E\to E$ be a normal $\G$-covering of $E$.
%

Let~$\cA$ be a finite von Neumann algebra with involution~${}^*$,
and let~$\tau\colon\cA\to\C$ be a finite, faithful, normal trace.
Let~$\ell^2(\cA)$ be the completion of $\mathcal A$ with respect to the scalar
product~$\langle\cdot,\cdot\rangle$ with~$\langle a,b\rangle=\tau(b^*a)$.

A right Hilbert $\mathcal A$-module is a Hilbert space~$M$ with a continuous
right $\cA$-action that admits an $\cA$-linear isometric
embedding into~$\ell^2(\cA)\otimes H$ for some Hilbert space~$H$;
this embedding is not part of the structure.
$M$ is \emph{finitely generated} if one can choose a finite-dimensional~$H$.
A {\em right Hilbertian $\mathcal A$-module\/}
is a topological vector space~$M$ with a continuous
right $\cA$-action such that there exists compatible scalar products
on~$M$ that turn~$M$ into a right Hilbert module.
Every such scalar product is called {\em admissible.\/}

\begin{rem}\label{MetricRem}
  If~$\langle\cdot,\cdot\rangle$ is an admissible scalar product,
  then all other admissible scalar products are
  of the form~$\langle S\cdot,\cdot\rangle$,
  where~$S$ is a self-adjoint, positive, invertible endomorphism of~$M$
  that commutes with the action of~$\cA$.
  In particular,
  the space of admissible scalar products is always contractible.
  All admissible scalar products give rise to isomorphic Hilbert modules,
  but the corresponding isomorphisms are not canonical.
\end{rem}

We denote with $\mathcal{B}_\mathcal{A}(M)$ the von Neumann algebra of bounded
$\mathcal A$-linear operators on $M$. The (unbounded) trace on $\mathcal{B}_\mathcal{A}(M)$ induced by $\t$ and by the usual trace on $\mathcal{B}(H)$ is denoted by  $\trt$, and $\mathcal{B}^1_\mathcal{A}(M)$ is the \emph{ideal of trace class operators}.

\subsubsection{Flat bundles of $\mathcal A$-modules}
\label{fbA}
We now fix a finitely generated Hilbertian $\Gamma$-$\cA$ bimodule~$M$,
in other words, a finitely generated right Hilbertian module
that admits a commuting $\Gamma$-action from the left.
We then consider the bundle of $\cA$-modules
\begin{equation}\label{coeff-bundle}
  \cF=\widetilde E\times_\Gamma M\to E\;.
\end{equation}
This bundle comes equipped with a natural flat $\cA$-linear
connection~$\nabla^{\cF}$.
The space of $\cF$-valued
smooth differential forms~$\Omega^\bullet(E;\cF)$
becomes a cochain complex with differential the usual extension of $\nabla^{\cF}$ to forms, which we will denote $d^E$ as in \cite{BL}.

Let $TZ$ be the vertical tangent bundle of $p$. We  fix a horizontal subbundle~$T^HE\subset TE$
such that~$TE=TZ\oplus T^HE$. If $U$ is a smooth vector field on $B$, let $\bar U\in C^\infty(E, T^HE)$ denote its horizontal lift so that $\pi_*\bar U=U$, and $P^{TZ}$ be the projection from $TE$ to $TZ$. 
This defines an isomorphism
\begin{equation}\label{dEdef}
  \Omega^\bullet(E;\cF)
	\cong\Omega^\bullet\bigl(B;\Omega^\bullet(E/B;\cF)\bigr)\ .
\end{equation}
Let $\mathcal W\to B$ be the smooth infinite-dimensional $\mathbb Z$-graded
bundle over $B$ whose fibre is $\mathcal W^\bullet_b=C_0^\infty(Z_b,
(\Lambda^\bullet (T^*Z)\otimes \mathcal F){}_{Z_b})$, the compactly supported fibrewise smooth differential forms.
 
 Let $g^{TZ}$ be a vertical metric, and $g^{\cF}$ a smooth family of admissible scalar products on the bundle~$\cF\to E$.
This induces a family of $L^2$-metrics~$g^{\mathcal W}$
on $\mathcal W$. %the infinite dimensional bundle~$\Omega_0^\bullet(E/B;\cF)\to B$
%of compactly supported fibrewise smooth differential forms.
The fibrewise $L^2$-completion of~$\mathcal W$
is denoted~$\Omega_{L^2}^\bullet(E/B;\cF)\to B$.
As a Hilbert space,
it is isomorphic to the Hilbert tensor
product~$\Omega^\bullet_{L^2}(E/B)\otimes M$,
and the topology of~$\Omega_{L^2}^\bullet(E/B;\cF)\to B$
is independent of the choice of admissible metrics above.
Thus we can regard~$\Omega_{L^2}^\bullet(E/B;\cF)\to B$
as a locally trivial bundle of Hilbertian $\cA$-modules,
with a family of admissible metrics~$g^{\mathcal W}$.
%Since $\Omega_{L^2}^\bullet(E/B;\cF)$ is not a $C^\infty$ bundle, 
We will define connections and do analysis on the subbundle $\mathcal W$.
%, which will be denoted $\mathcal W\to B$ in the following.

The fibrewise derivative~$d^Z=\nabla^{\cF}|_{TZ}$
becomes an unbounded operator
\begin{equation*}
  d^Z\colon\Omega_{L^2}^k(E/B;\cF)
	\longrightarrow\Omega_{L^2}^{k+1}(E/B;\cF)
\end{equation*}
which can be seen as an element of $C^\infty(B, \Hom (\mathcal W^\bullet, \mathcal W^{\bullet+ 1}))$.

As shown in \cite[III.(b)]{BL}, the connection~$d^E=\nabla^{\cF}$ now becomes a flat $\cA$-linear
superconnection
\begin{equation}\label{dE}
  d^E=d^Z+\nabla^{\mathcal W}+\iota_T
\end{equation}
of total degree~$1$ on the bundle~$\mathcal W \to B$.
Here $\nabla^{\mathcal W}:=\mathcal{L}{}_{\,\bar{\dot}\,}$ is the Lie derivative
with respect to horizontal lifts,
~$T$ is the fibre bundle curvature of~$T^HE$ defined by $T(U, V)=-P^{TZ}[\bar U, \bar V]$, and $\iota_T$ is the interior multiplication by $T(\cdot, \cdot)$. 

\subsubsection{The fibrewise $L^2$-cohomology}\label{L2coh}
We consider the {\em reduced $L^2$-cohomology\/}
\begin{equation}\label{H2}
  H_{L^2}^\bullet(E/B;\cF)
	=\Ker d^Z/\overline{\Im d^Z}
\end{equation}
and we obtain therefore a bundle $\mathcal A$-Hilbert modules \cite[1.4.2]{Lu}.

Because~$d^E%=\nabla^{\cF}
$ is a flat superconnection, the connection~$\nabla^{\mathcal W}$ induces
a connection~$\nabla^{\Ker d^Z}$ on the bundle~$\Ker d^Z\to B$
such that~$\Im d^Z$ is a parallel subbundle (as in \cite[p. 307]{BL}).
If the bundle~$E\to B$ is trivial and~$g^{TZ}$, $g^{\cF}$
and~$T^HE$ are of product type,
then~$\overline{\Im d^Z}$ is clearly also parallel,
and~$\nabla^{\Ker d^Z}$ induces a flat $\cA$-linear connection~$\nabla^H$
on the bundle~$H^\bullet_{L^2}(E/B;\cF)$.
%Because
%the topology on~$\Omega(E/B;\cF)$ does not depend on the choice
%of~$g^{TZ}$, $g^{\cF}$ and~$T^HE$
%and the connection~$\nabla^{\ker d^Z}$ is independent on~$T^HE$ modulo
%sections of~$\Im d^Z$,
As in Bismut--Lott \cite{BL}, it turns out that the connection~$\nabla^H$ is well-defined
and independent of the choices of~$g^{TZ}$, $g^{\cF}$
and~$T^HE$ in the case that~$E\to B$ is a product bundle.
Because the bundle~$E\to B$ is assumed to be locally trivial,
we obtain naturally a reduced \emph{Gau\ss--Manin connection}~$\nabla^H$
on the reduced $L^2$-cohomology $H^\bullet_{L^2}(E/B;\cF)\to B$.
This connection is still $\cA$-linear and flat.

The operator~$(d^Z+d^{Z,*})^2$ is the fibrewise Hodge-Laplacian,
and by Hodge theory the reduced $L^2$-cohomology is given by
\begin{equation}\label{isoHodge2}
  H^\bullet_{L^2}(E/B;\cF)\cong\ker(d^Z+d^{Z,*})=\ker(d^Z+d^{Z,*})^2
  \subset\Omega^\bullet_{L^2}(E/B;\cF)\;.
\end{equation}
Restriction of~$g^{\mathcal W}$
to~$H^\bullet_{L^2}(E/B;\cF)$ thus defines
an $L^2$-metric~$g^H_{L^2}$.
As the restriction of an admissible metric to 
an $\cA$-invariant
subbundle,
the metric~$g^H_{L^2}$ is also admissible.
Moreover the fibres of~$H^\bullet_{L^2}(E/B;\cF)$ are finitely
generated as Hilbertian $\cA$-modules by \cite{Sh}.

\begin{ex}\label{ex1.1}

  Let $\pi \colon \tilde E\to E$ be a normal $\G$-covering of the fibre bundle $p\colon E\to B$ as in Section \ref{Sect.NCF}.
  Let~$\ell^2(\G)$ be the completion of the group ring~$\C\G$
  with respect to the standard $L^2$-scalar product.
  The group von Neumann algebra~$\mathcal N(\Gamma)$ of~$\Gamma$
  consists of all bounded operators on~$\ell^2(\Gamma)$ that commute
  with the left regular representation of~$\G$.
  It contains~$\C\G$ as a weakly dense subset,
  and on~$\C\G$,
  the canonical trace~$\tau$ is given by
  \begin{equation*}
    \tau\Bigl(\sum a_\gamma\,\gamma\Bigr)=a_e\;.
  \end{equation*}
  Then~$M=\ell^2(\Gamma)$ is a finitely generated Hilbertian $\G$-$\mathcal N \G$-
  bimodule, indeed $M\simeq l^2(\mathcal N\G)$.

  We fix a fibrewise Riemannian metric~$g^{TZ}$ on~$TZ$.
  Because the standard $L^2$-scalar product on~$\ell^2(\Gamma)$
  is $\G$-invariant,
  it defines a natural family of admissible scalar products~$g^{\cF}$
  on~$\cF=\tilde E\times_\G \ell^2(\G)$.
  We now have a natural $\mathcal N(\G)$-linear
  isometric isomorphism
  \begin{equation*}
    \Omega_{L^2}^\bullet(\widetilde E/B)
    \cong\Omega_{L^2}^\bullet(E/B,\cF)
  \end{equation*}
  that is compatible with the flat superconnection~$d^E$ of~\eqref{dEdef}.
  In particular,
  the flat Hilbertian $\mathcal N(\G)$-module
  bundle~$H^\bullet_{L^2}(E/B;\cF)\to B$
  with the Gau\ss--Manin connection is isomorphic to the
  fibrewise $L^2$-cohomology of the normal covering~$\widetilde E\to E$.
\end{ex}

\begin{ex}\label{ex1.2}
  If~$\Gamma$ acts on a vector space~$V$,
  there exists a flat vector bundle
  \begin{equation*}
    F=\widetilde E\times_\G V\;.
  \end{equation*}
  We now consider~$M=V\otimes\ell^2(\Gamma)$ with the diagonal $\G$-action.
  A $\G$-invariant metric on~$\pi^*F=\Et\times V\to E$ defines a family
  of admissible metrics~$g^{\cF}$
  on~$\cF=\widetilde E\times_\G M$
  and a metric~$g^F$ on~$F\to E$.
  We have a natural isometry of bundles of Hilbertian
  $\mathcal N(\G)$-modules
  \begin{equation*}
    \Omega^\bullet_{L^2}(\widetilde E/B;\pi^*F)
    \cong\Omega^\bullet(E/B,\cF)
  \end{equation*}
  that is compatible with~$d^E$ as above.
%  Note that this complex is isomorphic to the one in Example~\ref{ex1.1}
%  as a complex of Hilbertian spaces if~$\rk V=1$,
%  but in general, the $\mathcal N(\Gamma)$-action is different,
%  and hence also the set of admissible metrics.
\end{ex}

\subsection{Families of normal coverings}\label{Sect.FNC}
Let~$p\colon E\to B$ be a smooth proper submersion,
and assume that there exists a covering\footnote{not necessarily normal.} map~$\pi\colon\widetilde E\to E$
such that over each point~$b\in B$, the map~$\pi_b\colon\widetilde E_b\to E_b$
is a normal covering.
Then the groups of covering transformations form a locally trivial bundle of discrete groups over~$B$
that is in general nontrivial.

\begin{dfn}\label{def1.1}
  A {\em family of normal coverings\/}~$(\widetilde E,\G)\to B$
  of~$p\colon E\to B$ consists of a bundle of discrete groups~$\Gamma\to B$
  and a covering~$\pi\colon\widetilde E\to E$
  such that~$\pi_b\colon\widetilde E_b\to E_b$ is a normal covering
  with group of covering transformations~$\Gamma_b$ for all~$b\in B$
  in a continuous way.
\end{dfn}
\begin{ex}\label{Bundleoftori}
A typical non-trivial example for this situation arises from a flat vector
bundle $\tilde E\to B$ with structure group $\SL_n(\mathbb{Z})$.
Any such a flat vector bundle is associated to a principal
$\SL_n(\mathbb{Z})$-bundle $P\to B$. The action of $\SL_n(\mathbb{Z})$ on
$\mathbb{Z}^n$ by group automorphisms also gives rise to an associated
non-trivial bundle of groups $\Gamma\to B$ with fibers isomorphic (in a non-canonical way) to
$\mathbb{Z}^n$. The fibers of $\Gamma$ act
in a canonical way on the fibers of $\tilde E$ by deck transformations. The fiberwise
quotient produces a (in this case flat) bundle of tori, the one 
associated in the canonical way to $P$.
\end{ex}
\begin{dfn}\label{def1.2}
  Let~$(\widetilde E,\G)\to B$ be a family of normal coverings
  of~$p\colon E\to B$,
  and let~$\cA\to B$ be a locally trivial family of von Neumann
  algebras over~$B$ with discrete structure group.
  A {\em family of Hilbertian $\G$-$\cA$-bimodules\/}
  is a locally trivial family of Hilbertian spaces~$M\to B$ with discrete
  structure group such that~$M_b$
  is a Hilbertian $\G_b$-$\cA_b$-bimodule for all~$b\in B$
  in a continuous way.
  We say that~$M$ is a {\em family of finitely generated
  Hilbertian $\G$-$\cA$-bimodules\/} if~$M_b$ is a finitely
  generated Hilbertian $\cA$-module for all~$b\in B$.
\end{dfn}

In both definitions,
``in a continuous way'' means that over $B$ we have local trivialisations of all the structure, including $\G_b\to \Diffeo (E_b)$. 
% over small
%open subsets of~$B$,
%the respective action is trivial, too.
For such actions,
we will henceforth simply write
``$\G$-actions'' and ``$\cA$-actions'' unless this could cause confusion.
Because we have fixed discrete structure groups,
both~$\cA\to B$ and~$M\to B$ are equipped with
natural flat connections.

Let~$M$ be a family of finitely generated Hilbertian $\G$-$\cA$-bimodules
and consider
\begin{equation}\label{coeff-bundle2}
  \cF=\widetilde E\times_\G M\longrightarrow E\;,
\end{equation}
which is equipped with a natural $p^*\cA$-action
and a $p^*\cA$-linear flat connection~$\nabla^{\cF}$.
Here ``$p^*\cA$-linearity'' means that
\begin{equation*}
  \nabla^{\cF}(as)=\bigl(\nabla^{p^*\cA}a\bigr)\cdot s+a\cdot\nabla^{\cF}s
\end{equation*}
for all sections~$a$ of~$p^*\cA\to E$ and~$s$ of~$\cF$.

Because locally
we are in the same situation as in Section~\ref{Sect.NCF},
we can repeat all constructions as before.
Thus, we construct a family of Hilbertian
$\cA$-modules $\Omega^\bullet_{L^2}(E/B;\cF)\to B$
carrying a flat superconnection~$d^E$ and
define the fibrewise reduced
$L^2$-cohomology $H^\bullet_{L^2}(E/B;\cF)\to B$.
Again, this is a family of finitely generated Hilbertian $\cA$-modules.

\begin{rem}
The reader should keep in mind that all analytic manipulations take place
along the single fibre of $p$, and will depend only on the geometry of this
fibre inside $E$.
Hence, the two settings of Sections \ref{Sect.NCF} and \ref{Sect.FNC} work equally well.
\end{rem}

%\subsection{Main examples}\label{Sect1.3}
%We show how the standard examples of flat $\G$ -equivariant vector bundles
%over $\G$-normal coverings fit into the framework of subsection~\ref{Sect.NCF}.

\begin{rem}Let~$p\colon E\to B$ be an arbitrary smooth proper submersion
with fibre~$Z$.
From the long exact sequence
\begin{equation*}
  \begin{CD}
    \pi_2(B)@>>>\pi_1(Z)@>>>\pi_1(E)@>>>\pi_1(B)
  \end{CD}
\end{equation*}
we see that in general
there exists no normal subgroup of~$\pi_1(E)$ with quotient~$\pi_1(B)$.
In particular,
it is in general not possible to take the fibrewise universal covering
globally.
\end{rem}
\begin{ex}\label{ex1.3}
%  Assume that~$p\colon E\to B$ admits
%  a covering~$\pi\colon\widetilde E\to E$
%  such that~$\pi_b\colon\widetilde E_b\to E_b$ is a normal covering
%  with group~$\Gamma_b$ of covering transformations for all~$b\in B$.
  If~$p$
  admits a section~$e_0\colon B\to E$ then such a covering can be constructed.
  In this case, we consider the fibrewise universal covering $\widetilde E$,
  which consists of all fibrewise paths starting
  at~$e_0$ up to fibrewise homotopies that preserve the endpoints.
  Then~$\Gamma_b\cong\pi_1(E_b,e_0(b))$ is the family of fibrewise
  fundamental groups with respect to~$e_0$.
  Note that~$\Gamma_b$ is not necessarily a trivial bundle,
  and that different sections~$e_0$ can produce non-isomorphic coverings, as in the following case.
 
 Consider for instance a closed oriented surface $F_3$ of genus $3$, which we
 regard as the gluing $ F_3=S_1\cup_{\gamma} S_2$ along a circle $\gamma$ of a surface $S_1$ of genus $1$ with one boundary component $\gamma$ together with a surface $S_2$ of genus two with one boundary component $\gamma^{-1}$.
 Let $\a\colon F_3\to F_3$ be the Dehn twist on a tubular neighborhood of $\g$, and let $p\colon \mathcal T_\a\to S^1$ be the mapping torus fibre bundle over $S^1$ with $\mathcal T_\a=\left(F_3\times [0,1]\right)/\sim$, where $(x, 0)\sim (\a(x),1)$. Let $P_1\in S_1$ and $P_2\in S_2$ be two base points on $F_3$, both fixed by $\a$. Then $p$ has the two global sections given by $s_i([t])=[(P_i, t)]$, for $[t]\in S^1$, $i\in 1,2$. We can form two bundles $\widetilde E^i\to S^1$, $i=1,2$, of fibrewise universal coverings as explained above. The corresponding bundles of groups $\G^i\to S^1$, $\G^i_{[t]}=\pi_1(\mathcal T_\a{}_{[t]}, P_i)$, are the mapping tori of the maps $\a_*^i\colon \pi_1(F_3, P_i)\to  \pi_1(F_3, P_i)$. It is not difficult to see that $(\widetilde E_1, \G_1)$ and $(\widetilde E_2, \G_2)$ are not isomorphic: indeed there exists no isomorphism between $\pi_1(F_3, P_1)$ and $\pi_1(F_3, P_2)$ that intertwines $\a_*^1$ and $\a_*^2$, as they fix subgroups of different rank.

  As in Example~\ref{ex1.1},
  we can form a family~$\mathcal N(\Gamma)\to B$ of von Neumann algebras
  and a family~$\ell^2(\Gamma)\to B$ of finitely generated Hilbertian
  $\G$-$\cA$-bimodules.
  Since the situation is locally isomorphic to Example~\ref{ex1.1},
  we may then proceed as above.
\end{ex}

\section{\texorpdfstring{$L^2$}{L2}-invariants for families: superconnections, heat operator}
%analytic tools
\label{Sec:tools}

In this Section, we introduce our two problems and set up a unified formalism to treat both with similar methods in the rest of the paper.

Let $Z\to E\stackrel{p}{\to}B$ be a smooth fibre bundle with connected
$n$-dimensional closed Riemannian fibres, let $\mathcal F\to E$ be a bundle of
$\mathcal A$-modules as in \eqref{coeff-bundle} or \eqref{coeff-bundle2}.
Let $\{e_i\}_{i=1}^n$ be a local orthonormal framing of $TZ$. Exterior
multiplication by a form $\phi$ will be denoted by $\phi \w$, interior
multiplication by a vector $v$ will be denoted $i_v$.
As usual, we identify vertical tangent and cotangent vectors using the
fibrewise Riemannian metric and we denote for a vertical vector $X$
\begin{equation*}\label{Cliff}
c(X)=(X\w ) - i_X\;,\;\;\;\; \hat c(X)=(X\w ) + i_X
\end{equation*}
and put 
$$c^i=c(e_i) \;,\;\;\;\;  \hat c^i=\hat c(e_i)$$
 which generate two
graded-commuting Clifford module structures on forms (for the bundle of
Clifford algebras associated to the vertical tangent bundle), compare \cite[III.c]{BL}.

Let $N$ denote the \emph{number operator} on vertical forms, acting as $N\phi=p\phi$, for $ \phi \in C^\infty(E, \Lambda^p(T^*Z))$. We have then $\sum_{i=1}^n c^i \hat c^i=2N-n$.

\subsection{Trace norm and spectral density function}
Let $\End^1_{\mathcal A} \mathcal{W}\rightarrow B$ denote the bundle whose sections are families of $\tau$-trace class operators.
%Smooth sections $T\in C^\infty(B, \End^1_{\mathcal A} \mathcal{W})$ are those whose Schwartz kernel is in $C^\infty(E\times_\pi E, (\Lambda T^*(E/B)\otimes \mathcal F)\otimes_\pi  (\Lambda T^*(E/B)\otimes \mathcal )F )^*$
%If $T\in C^\infty(B, \End^1_{\mathcal A} \mathcal{W})$, then $\trt (T)\in C^\infty(B)$, and if $T\in \O(B, \End^1_{\mathcal A} \mathcal{W}) $, then $\trt T\in\O(B)$.
Equip the ideal $\End^1_{\mathcal A} \mathcal{W}$ with the norm 
$$
\n{A}_\t:=\trt(|A|)\ .
$$

Let $D=\frac{1}{2}(d^{Z,*}-d^Z)$ (note that, here and throughout the whole paper, $D$ is a skew-adjoint operator and $D^2\leq 0$),  and let $e^{tD^2}$ be the fibrewise heat operator associated to the Hodge Laplacian. The following proposition by Gong and Rothenberg is fundamental for what follows. 
\begin{prop}\cite{GR}
\label{SmoothProp} Let $P=(P_b)_{b\in B}$ be the family of projections onto $\ker D$. Then
$$
P\in C^{\infty} (B;\End{}_{\mathcal A}^1\mathcal{W})
\;\text{ and }\;e^{tD^2}\in C^{\infty}( B, \End{}_{\mathcal A}^1\mathcal{W}) \ .$$
\end{prop}
\begin{proof}
Proved in \cite[Lemma 2.2 and Theorem 2.2]{GR}, see also \cite[Theorem 4.4]{BH4}. 
\end{proof}
We denote
\begin{equation}
\label{NStheta}
\theta_b(t):=\trt (e^{tD_b^2}-P)\end{equation}
or simply $\theta(t)$, as the dependence on the base point will not be crucial.
By results of Gromov--Shubin \cite{GS}, the dilatation class of $\theta$ as $t\rightarrow \infty$ is known to be a homotopy invariant: if $b, \bar b \in B$, then $\exists C_b$ with $\theta_b(C_b t)\leq \theta_{\bar b}(t)\leq \theta_b(\frac{t}{C_b})$. Moreover, in the proofs in \cite{GS} and \cite{Ef} one can choose constants $C_b$ (constructed from the chain homotopy equivalence) depending in a continuous way on $b\in B$. This implies that there is uniformity on compact subsets of $B$.

It is clear that  
$ \lim_{t\rightarrow \infty}\theta_b(t)=0$.
More precisely, we even have, uniformly on compact subsets of $B$,
$$
\lim_{t\rightarrow \infty}e^{tD^2}=P\in \O^\bullet(B;\End{}_{\mathcal A}^1\mathcal{W})
$$
in the trace norm. 
The operators $e^{tD^2}$ strongly converge to $P$. However, if $0$ is in the continuous spectrum of $D$, then $e^{tD^2}$ does not converge in the operator norm topology.
  
  \bigskip

Corresponding to the different push-forward theorems we want to prove, we will have two types of flat bundles $\mathcal F$ of $\mathcal A$-modules.
\begin{description}
	\item[Setting A] $\mathcal F\to E$ is a flat bundle of $\mathcal
          A$-modules. This is the setting to prove the $L^2$-version of the
          Bismut--Lott theorem and to construct the $L^2$-torsion form.
	\item[Setting B] $\mathcal F\to E$ is a \emph{flat duality bundle} of $\mathcal A$-modules. This is the natural setting to discuss the signature operator, as observed in \cite{Lo2}.
\end{description}

\subsection{Flat bundles, $L^2$-Kamber--Tondeur classes}
\label{flKT}
\subsubsection{$L^2$-Kamber--Tondeur classes}
\label{L2KamTon}
Let $\cF\rightarrow E$ be a flat bundle of $\mathcal A$-Hilbert modules, with flat connection $\nabla^\cF$. If $g^\cF$ is a scalar product on $\cF$, let $\nabla^{\cF,*}$ be the adjoint connection, and put $\omega(\cF, g^\cF):=\nabla^{\cF,*}-\nabla^\cF=(g^\cF)^{-1}\left(\nabla^\cF g^\cF\right)\in \Omega^1(E, \End\cF)$. Using the trace $\t$ as in \cite[Sec. 4]{Sc1}, we define 
$$c_{k, \tau}(\mathcal{F}, g^\mathcal{F}):=(2\pi i)^{-\frac{k-1}{2}}\trt\left(\frac{\omega(\cF, g^\cF)}{2}\right)^ k\in \O^k(E)
$$
to be the \emph{$L^2$-Kamber--Tondeur forms}. They are closed forms, and the
corresponding \emph{$L^2$-Kamber--Tondeur classes} in $H_{dR}^*(E)$ do not
depend on the metric $g^\cF$. 
Let 
$$
\ch{}^\circ_\tau (\mathcal{F}, g^\mathcal{F}):=\sum_{j=0 }^\infty
\frac{1}{j!}c_{2j+1,\tau}(\nabla^\mathcal{F}, g^\mathcal{F})=\frac{1}{\sqrt
  {2\pi i}}\Phi\trt\left( \frac{\omega(\cF, g^\cF)}{2} e^{ (\frac{\omega(\cF,
      g^\cF)}{2})^2}\right)\; \in \O^*(E)
$$
where $\Phi(\a)= (2\pi i)^{-\frac{|\a|}{2}}\a$, and denote its cohomology class by $\ch^\circ_\tau (\mathcal{F})\in H^*_{dR}(E)$. 
The classes $\ch{}^\circ_\tau (\mathcal{F})$ vanish whenever $\cF$ admits a $\nabla^\cF$-parallel metric.
For a $\mathbb Z_2$-graded bundle the Kamber--Tondeur class is defined using the corresponding supertrace.

\subsubsection{Superconnection formalism}
\label{uni-treat}
Let $Z\to E\stackrel{p}{\to}B$ be a smooth fibre bundle with connected $n$-dimensional closed fibres, let $\mathcal F\to E$ be a bundle of $\mathcal A$-modules as in \eqref{coeff-bundle} or \eqref{coeff-bundle2}.

As seen in Sections \ref{L2coh} and \ref{Sect.FNC}, the fibrewise
$L^2$-cohomology with coefficients in $\cF$ has the structure of a flat bundle
of $\cA$-modules $H^\bullet_{L^2}(E/B;\cF)\to B$, which we consider as the
\emph{analytic push-forward} of $\cF$. We will compute their Kamber--Tondeur
classes $\ch{}^\circ_\tau(H_{L^2}(E/B;\cF))$ in the push-forward Theorem
\ref{L2BL} which will make use  of the following superconnection formalism.

The infinite dimensional bundle $\cW\to B$ defined in \ref{fbA} is endowed with the $L^2$-metric 
\begin{equation*}
g^\cW_b(\phi\otimes f,\phi'\otimes f')= \int_{Z_b}\phi\w * \overline{\phi'} \cdot g^\cF(f,f')\ .
\end{equation*}
Consider the $\mathbb Z_2$-grading on $\cW$ induced by the degree of vertical forms. We denote it by $\cW=\cW^0\oplus \cW^1\to B$, and call it the \emph{de Rham, or Euler grading}. 

Let $d^{E,*}$ be the \emph{adjoint superconnection} of~$d^E$ of \eqref{dE} with respect
to~$g^{\cW}$ in the sense of \cite[I.d]{BL}, then as in \cite[Prop. 3.7]{BL}  we have
\begin{equation}\label{dE*}
  d^{E,*}=d^{Z,*}+\nabla^{\cW,*}+\ep_T
\end{equation}
where $d^Z$ is the fibrewise formal adjoint of $d^Z$ with respect to $g^\cW$, $\nabla^{\cW,*}$ is the adjoint connection of $\nabla^{\cW}$, and $\ep_T=\iota_T^*$.
Define 
\begin{equation}
\label{AX}
\mathbb{A}=\frac{1}{2}(d^E+d^{E,*})\;\;;\;\;\;\;
\mathbb{X}=\frac{1}{2}(d^{E,*}-d^E)\ .
\end{equation}
$\mathbb A$ is a superconnection on $\cW^0\oplus \cW^1\to B$.
We denote by $\O(B, \mathcal{W})^{dR}=\O(B,\mathcal{W}^{0}\oplus \mathcal{W}^{1})$ the graded tensor product algebra between sections of $\cW$ and  differential forms on the base. This $\mathbb Z_2$-grading defines on $\End^1_\cA \cW$ the supertrace 
\begin{equation}\label{strdR}
\Str_\t T=\tr ((-1)^N T)\ .
\end{equation}

\begin{rem}\label{remX}
Because $d^{E,*}-d^E$ is the difference of two superconnections, $\mathbb{X}$ is an (odd) element of $\Omega^\bullet(B, \End \mathcal W)^{dR}$, and in particular  it differentiates only along the fibres. 
\end{rem}
Perform the usual rescaling
\begin{equation}
\mathbb{A}_t:=\frac{1}{2} t^{\frac{N}{2}} \left(d^{E}+d^{E,*} \right)  t^{-\frac{N}{2}}\;\;\;, \;\;\;\;
\label{AX-resc1}
\mathbb{X}_t:=\frac{1}{2} t^{\frac{N}{2}} \left(d^{E,*}-d^{E} \right)  t^{-\frac{N}{2}}\ .
\end{equation}
where $N$ is the number operator of $\mathcal W$.
We have 
%$
\begin{equation}
\mathbb{A}_t=\frac{\sqrt t}{2}(d^Z+d^{Z,*})+\nabla^u-\frac{c(T)}{4\sqrt t}
\end{equation}
%$
and 
%$
\begin{equation}
\mathbb{X}_t=\frac{\sqrt t}{2}(d^{Z,*}-d^Z)+\frac{\o}{2}-\frac{\hat{c}(T)}{4\sqrt t}
%$, 
\end{equation}
where $\nabla^u=\frac{1}{2}(\nabla^\cW+\nabla^{\cW,*})$, and 
\begin{equation}\label{omega}
\o=\nabla^{\cW,*}-\nabla^\cW\ .
\end{equation}
Let $f(a)= ae^{a^2}$. 
For $ \alpha \in \O(B)$,  put $\Phi \alpha=(2\pi i)^{-\frac{\deg \alpha}{2}}\alpha$,  and define 
      \begin{equation}\label{Ftau}
      F_\tau(t)=\sqrt{2\pi i}\,\Phi\Str_\t\left(f(\mathbb{X}_t)\right)\;\; \in \O(B)\ .
        \end{equation}  
    It follows, as in \cite[Theorem 1.8]{BL}, that $ F_\tau(t)$ is a real closed odd form.

\subsubsection{Transgression terms}
The transgression of \eqref{Ftau}, and later of the heat operator's
supertrace, can be computed as  as in \cite[p. 311]{BL}, and is given by
      \begin{equation}\label{trF^}
   \frac{d}{d t} F_\t(t) =\frac{1}{t} dF_\tau^{\wedge}(t)
        \end{equation}
where 
    \begin{equation}\label{Fhattau}
F_\tau^{\wedge}(t)=\Phi \Str_\t\left(\frac{N}{2}(1+2\mathbb X_t^2) e^{\mathbb X_t^2}\right)\ .
        \end{equation}  
Equivalently, one can proceed as follows.
Let $\breve B=B\times \mathbb{R}^*_+ $, where $\mathbb{R}^*_+$ denotes the time direction. We fix an arbitrary metric $g^B$ on $B$. 
Define $\breve{\pi} \colon \breve{E}\rightarrow \breve{B}$ where  $\breve{E}=E\times \mathbb{R}^*_+$ and  $\breve{\pi}(x,t)=(\pi(x),t)$. Endow $\breve{\pi}$ with the vertical metric \begin{equation}
\label{t-resc}
g^{T\breve{Z}}_{|t}=\frac{g^{TZ}}{t}
\end{equation}
and on the base take $g^{\breve{B}}=g^B\oplus g^{\mathbb{R}^*_+}$. We have
$
g_t^{\mathbb{R}^*_+}\left(\frac{\partial}{\partial t},\frac{\partial}{\partial t}\right)=\frac{1}{t^2}
$.  For simplicity let $d\vartheta=\frac{dt}{2t}$.
On $\breve{E}$ we have
$d^{\breve{E}}=d^E+\frac{\partial}{\partial t} dt= d^Z+ \nabla^\mathcal{W}+dt\frac{\partial}{\partial t}+i_{\rm T}$, and 
its adjoint superconnection is
$
d^{\breve{E}, *}=td^{Z,*}+\nabla^{\mathcal{W},*}+dt\frac{\partial}{\partial t}+\left(N-\frac{n}{2}\right)\frac{dt}{t}-\frac{\epsilon_{\rm T}}{\sqrt t}
$, because the fibrewise rescaling \eqref{t-resc} gives $g^{\mathcal{\breve{W}}^i}{}_{|t}={\sqrt t}^{2i-n}g^{\mathcal{W}^i}$.
We define (deviating here from the notation in \cite[III. (i)]{BL})
\begin{equation}
\mathbb{\breve{A}}:=\frac{1}{2} t^{\frac{N}{2}} \left(d^{\breve{E}}+d^{\breve{E},*} \right)  t^{-\frac{N}{2}}\;\;;\;\;\;\;
\label{X-hat}
\mathbb{\breve{X}}:=\frac{1}{2} t^{\frac{N}{2}} \left(d^{\breve{E},*}-d^{\breve{E}} \right)  t^{-\frac{N}{2}}
\end{equation}
and we have 
\begin{equation}\label{A^}
\mathbb{\breve{A}}
=\mathbb{A}_t+d\vartheta\left(\frac{\partial}{\partial t}+(N-\frac{n}{2})\right)\;\;
\end{equation}  
\begin{equation}
\label{X^}
\breve{\mathbb X}=\mathbb X_t+\left(N-\frac{n}{2}\right)d\vartheta\ .
\end{equation}
%Then $\mathbb{\breve X}=\sqrt t D+\hat R_t$,  where 
%$ \breve R_t=R_0+t^{-\frac{1}{2}}R_{1}+(N-\frac{n}{2})\frac{dt}{2t} $\ . 

%Then $\eta_\t(t)$ of \eqref{etaevenint} is the $d\vartheta$-term in $e^{-\hat{\mathbb A}^2}$.

\subsection{Flat duality bundles}
Flat duality bundles (over the real numbers) have been introduced by Lott and
further investigated by Bunke and Ma as cycles for a certain
$\mathbb Z/2$-graded homotopy
invariant contravariant functor $\bar L^*(X)$ \cite{Lo2, BM}.

Deviating from the notation employed by \cite{Lo2, BM},
it seems reasonable to think of the groups  $\bar L^*(X)$ as the degree $0$ and degree $2$
part of a $4$-graded group. 
 For us, the main feature is that they can be
paired with the signature operator of an even dimensional oriented manifold
%(which might be thought of as a fundamental class of a dual covariant
%functor), 
and we study push-forward theorems in the $L^2$-context.

Following Lott, we introduce flat duality bundles of $\cR$-modules, where $\cR$ is a (real) finite von Neumann algebra in the sense of \cite{Ay}. Nevertheless, since our focus is on the signature operator acting on complex differential forms, in the applications we shall mainly use complex von Neumann algebras, or pass to the algebra $\cR\otimes_{\mathbb R} \C$ generated by $\cR$.

\medskip

Let $\mathcal R$ be a real finite von Neumann algebra. Let $\mathcal F\to E$ be a bundle of finitely generated $\mathcal R$-Hilbert modules over $E$. Let $\ep\in \{-1, +1\}$.
\begin{dfn}
\label{fdb}
$\mathcal F$ is called a \emph{flat duality bundle of $\mathcal R$-modules} if it is endowed with a flat connection $\nabla^\mathcal F$ and a bilinear form $Q^\mathcal F\colon \mathcal F\otimes \mathcal F\to \R$ such that 
\begin{itemize}
\item[i)] $Q^\mathcal F$ is $\ep$-symmetric, \emph{i.e.} $Q^\mathcal F(y, x)=\ep Q^\mathcal F(x, y)$;
\item[ii)] $Q^\mathcal F$ is non-degenerate (i.e.~$Q^\mathcal F $ is
  invertible as a map to the topological dual);
\item[iii)] $Q^\mathcal F(xa, y)=Q^\mathcal F(x,ya^*)$, $\forall x,y\in \mathcal F$, $\forall a\in \mathcal R$;
\item[iv)]  $
\nabla^\mathcal{F}Q^\mathcal{F}=0\ .
$
\end{itemize}
\end{dfn}
As in the finite dimensional case, one can always reduce the structure group:
\begin{lem}\label{redstructure}
Let $(\mathcal F\to E, \nabla^\cF, Q^\cF)$ be a flat duality bundle of $\mathcal R$-modules. Then there exists $J^\cF$ such that
$(J^\cF)^2=\epsilon$, $Q^\mathcal F(J^\cF x, J^\cF y)=Q^\mathcal F(x,y)$, and  $g^\cF(x, y)=Q^\mathcal F(x,J^\cF y) $ is a scalar product.
\end{lem}
\begin{proof}
Use polar decomposition, as in \cite[p.19]{BW}.
\end{proof}

\subsubsection{Characteristic classes of flat duality bundles}
We construct $L^2$-characteristic classes along the lines of \cite[Sec. 3.1]{Lo2}, as well as the formalism of flat duality superconnections.

For a flat duality bundle $\mathcal F\to E$ as in \ref{fdb}, fix  a scalar
product $g^\mathcal{F}(z,v)=Q^\mathcal{F}(z,J^\cF\, v)$. Let $\nabla^{\cF, *}$
be the adjoint connection with respect to $g^\mathcal{F}$  and define 
$\nabla^{\cF, u}=\frac{1}{2}(\nabla^\cF+\nabla^{\cF, *})$ which preserves $g^\mathcal{F}$.

\begin{dfn} If $\ep=1$, define
 $$p_\t(\nabla^\cF, J^\cF):=\tr_\t \left( J^\cF \cos \left(\frac{\omega(\mathcal F,g^{\mathcal F})^2}{8\pi}\right)\right) \in \O^{4\bullet}(E)$$
if $\ep=-1$,
 $$p_\t(\nabla^\cF, J^\cF):=-\tr_\t \left( J^\cF \sin \left(\frac{\omega(\mathcal F,g^{\mathcal F})^2}{8\pi}\right)\right) \in \O^{4\bullet+2}(E)\ .
$$  
\end{dfn}
If $\ep=1$, put $\Pi^\pm:=\frac{1\pm J^{\cF}}{2}$,  $\cF^\pm=\Pi^\pm\cF$, and $\nabla^{\cF^\pm}:=\Pi^\pm \nabla^{\cF, u} \Pi^\pm=\Pi^\pm \nabla^{\cF} \Pi^\pm$. 

  If $\ep=-1$, consider the complexified bundle $\cF_\C$, put $\Pi^\pm:=\frac{1\mp iJ^{\cF}}{2}$, define  $\cF^\pm:=\Pi^\pm\cF_\C$ and $\nabla^{\cF^\pm}:=\Pi^\pm \nabla^{\cF, u} \Pi^\pm$. 
In both cases we have, as in  \cite[Proposition 15]{Lo2},
\begin{equation}
\label{class.p}
p_\t(\nabla^\cF, J^\cF)=\ch_\t(\nabla^{\cF^+})-\ch_\t(\nabla^{\cF^-})\ .
\end{equation}

%\begin{rem}
%%Flat duality bundles (over the reals) have been introduced by Lott and
%%further investigated by Bunke and Ma as cycles for a certain
%%$\mathbb Z/2$-graded homotopy
%%invariant contravariant functor $\bar L^*(X)$ \cite{Lo2, BM}. 
%Deviating from the notation employed by \cite{Lo2, BM},
%it seems reasonable to think of the groups  $\bar L^*(X)$ as the degree $0$ and degree $2$
%part of a $4$-graded group. 
%%For us, the main feature is that they can be
%%paired with the signature operator of an even dimensional oriented manifold
%%(which might be thought of as a fundamental class of a dual covariant
%%functor). 
% Moreover, 
%there are maps to $K$-theory which are compatible with this 
%pairing. Then, Lott and Bunke--Ma define and study push-forwards in this theory
%along fibre bundles with even dimensional oriented fibres. For us the
%important point is that these indices are in some sense an enrichment of
%homology and therefore carry quite a bit of the rigid character of homology.
%It seems that the corresponding odd part of this
%theory has not been developped yet.
%\end{rem}

\begin{rem}
%One should note that a special case of a flat duality bundle is a flat bundle
%(which one equips with a trivial duality structure) - the corresponding
%Lott Chern character of the pairing with the signature operator is then the
%usual Chern character of the index of the signature operator twisted with the
%(flat) bundle.

Our flat duality bundles  of \ref{fdb} should be cycles in a variant of the
$\bar L$-groups defined by Lott and Bunke--Ma, which would have coefficients in the von
Neumann algebra 
$\mathcal{R}$. 
However, we do not develop this theory. Instead, we concentrate 
on the local $L^2$-index formula for the pairing with a family of signature
operators. 
\end{rem}

\subsubsection{Bilinear form on $\cW$ and superconnection formalism}
Recall that if $Z$ is a  closed oriented $n$-dimensional Riemannian manifold, the bilinear form on real differential forms $\O^\mathbb{R}(Z)$ defined by 
$$
Q^Z(\phi,\psi)= (-1)^{\frac{|\phi|(|\phi|-1)}{2}}\int_{Z}\phi\w\psi
$$
is $\ep_n$-symmetric, where $\ep_n=(-1)^{\frac{n(n+1)}{2}}$. Moreover the automorphism $J^Z$ defined on $\O^\mathbb{R}(Z)$ by
\begin{equation}\label{JZ}
J^Z\phi=(-1)^{\frac{|\phi|(|\phi|-1)}{2}}*\phi
\end{equation}
satisfies $(J^Z)^2=\ep_n$, and 
\begin{equation}\label{l2scalpro}
Q^Z(\phi,J^Z\phi')=\int_Z\phi\w*\phi'
\end{equation}
is the standard $L^2$-scalar product on forms \cite[Lemma 5]{Lo2}.
$Q^Z$ extends to a $\ep_n$-hermitian form on the $\C$-vector space  $\O(Z)$ of complex differential forms and the corresponding extension of \eqref{l2scalpro} gives the standard sesquilinear $L^2$-scalar product.

\medskip

Let now $Z\to E\stackrel{p}{\to}B$ be a smooth fibre bundle with connected
$n$-dimensional closed {oriented Riemannian} fibres, let $\mathcal F\to
E$ be a duality bundle of $\mathcal R$-modules constructed as in
\eqref{coeff-bundle} or \eqref{coeff-bundle2}, with $\ep$-symmetric bilinear
form $Q^\cF$  and flat connection $\nabla^\cF$.
By Lemma \ref{redstructure}, fix  a scalar product $g^\mathcal{F}(z,v):=Q^\mathcal{F}(z,J^\cF\, v)$, with $(J^\cF)^2=\ep$. 
%If $J^\cF$ is the identity, then $\nabla^\mathcal{F}g^\mathcal F=0$. 
%This will corresponds to the case $\mathcal{C}_0=0$ in \ref{AB} below.

On the infinite dimensional bundle $\mathcal W^{\mathbb R}\to B$ of vertical real differential forms with coefficients in $\cF$, the bilinear form 
\begin{equation}
Q_b^{\cW^\mathbb R}(\phi\otimes z, \psi\otimes \zeta):=(-1)^{\frac{|\phi|(|\phi|-1)}{2}}\int_{Z_b}\phi\w\psi \cdot Q^{\mathcal{F}}(z,\zeta)
\end{equation}
is $\ep\ep_n$-symmetric.
Let $J^Z$ the fibrewise automorphism defined by \eqref{JZ}. Then $J^{\cW}
(\phi\otimes z):=J^Z\phi \otimes J^\cF z$ satisfies $(J^{\cW} )^ 2=\ep\ep_n$.

The adjoint superconnection of~$d^E$ of  \eqref{dE*}
can be expressed as 
\begin{lem} \label{lemdE*} \cite[p. 328]{BL}, \cite[Proposition 30]{Lo2}.
$$
\begin{array}{l}d^{Z,*}=-(J^\cW)^{-1}\, d^Z\,J^\cW \\
\nabla^{\cW,*}=(J^\cW)^{-1}\,\nabla^{\cW}\, J^\cW\\
\iota_T ^*=-(J^\cW)^{-1} \,\iota_T\,J^\cW\ .
\end{array}
$$
\end{lem}
Define as before 
\begin{equation}
\label{AX1}
\mathbb{A}:=\frac{1}{2}(d^{E,*}+d^E)\;\;;\;\;\;\;
\mathbb{X}:=\frac{1}{2}(d^{E,*}-d^E)\ .
\end{equation}
and perform the usual rescaling as in \eqref{AX-resc1} to define $\mathbb{A}_t$ and $\mathbb{X}_t$.
In the language of \cite[3.2]{Lo2} the pair $(\mathbb{A},\mathbb{X})$ is a flat duality superconnection.

\begin{rem}
When working with complex differential forms, the quadratic form is extended as usual to a sesquilinear one, and we endow $\mathcal W$ with the metric $g^{\mathcal{W}}(u,v):=Q^\cW(u,J^\cW  v)$. Then consider the involution $J:=\frac{J^\cW}{\sqrt{\ep_n\ep}}$. The formulas of Lemma \ref{lemdE*} are still true with $J^\cW$ replaced by $J$. 
\end{rem}
Moreover we have, from  \cite[(3.36)]{BL} 
\begin{equation}\label{D.vs.sign}
d^Z+d^{Z,*}=c_j\nabla_{e_j}^{TZ\otimes \cF, u}-\frac{1}{2}\hat c_j \psi_j\ .
\end{equation}
where $\psi=\nabla^{\cF, *}-\nabla^\cF$ and  $\nabla^{TZ\otimes \cF, u}$
denotes the tensor product of the Levi-Civita connection on the vertical
tangent and  the unitary connection $\nabla^{\cF, u}:=\frac{1}{2}(\nabla^{\cF,
  *}+\nabla^\cF)$.

Define as usual (with an oriented frame $e_1,\dots, e_n$)
\begin{equation*}
\o_\C:=i^{\frac{n(n+1)}2} c_1 c_2\dots c_n\;\;,\;\;\;\;\;
\hat \o_\C:=i^{\frac{n(n+1)}2}\hat {c}_1\dots \hat{c}_n\ .
\end{equation*}
 $\o_\C$ is the chirality involution, related to the vertical Hodge star operator on $p$-forms by 
\begin{equation}\label{chir}
\omega_\C\,\phi=(-1)^{n p+\frac{p(p-1)}{2}}i^{\frac{n(n+1)}2} *\phi
\end{equation}
for $\phi\in C^\infty(E, \Lambda^p(T^*Z))$, and one has  $\,\hat\o_\C(-1)^N=\o_\C$ (see for example \cite[Lemmas 1.1.6 and 1.2.3]{Bo}).

\subsubsection{Transgression formulas, even dimensional fibres}
Let $\dim Z=2l$, and recall $\ep_n: =(-1)^{\frac{n(n+1)}{2}}$. 
We have  $\frac{J^Z}{\sqrt {\ep_n}}=\o_\C$ for $n=4j$, and $\frac{J^Z}{\sqrt {\ep_n}}=(-1)^{j+1}\o_\C$ for $n=4j+2$.  Then $J=\o_\C\otimes \frac{J^\cF}{\sqrt \ep}$.

Denote by $\cW=\mathcal W^+\oplus \mathcal W^-$ the grading defined by the
involution $J$, which we call the \emph{duality grading}. The graded tensor
product algebra between sections of $\cW$ and forms on the base will be
denoted  $\O(B, \mathcal{W})^{dua}=\O(B,\mathcal{W}^{+}\oplus
\mathcal{W}^{-})$, in contrast to the Euler grading
 $\O(B, \mathcal{W})^{dR}=\O(B,\mathcal{W}^{0}\oplus \mathcal{W}^{1})$
 considered in Section \ref{flKT}.
 
The operator $\mathbb A$ is a superconnection on $\O(B, \mathcal{W})^{dua}$. 
\begin{rem}
If $\mathbb B_t$ denotes the \emph{Bismut superconnection} for the family of signature operators with coefficients in the bundle $(\mathcal F, \nabla^{\cF, u})$, defined in \cite[10.3]{BGV}, then by \cite[Rem. 3.10]{BL}
\begin{equation}\label{AB}
  \mathbb A= \mathbb B_{\frac{1}{4}}-\frac{1}{4}\hat c_j \psi_j \ .
\end{equation}
Note that  $\frac{1}{4}\hat c_j \psi_j$ is a zero order operator vanishing if
and only if $g^\mathcal F$ is covariantly constant.  If $g^\cF$ is covariantly
constant, $d^Z+d^{Z,*}$ is  the \emph{signature operator} twisted by the
flat $\mathbb Z/2\mathbb Z$-graded bundle $\mathcal F=\mathcal F^+\oplus
\mathcal F^-$. 
\end{rem}

\begin{rem}
Because $\dim Z=2l$, we have $$\o_\C=(-1)^N \hat \o_\C$$ hence 
\begin{equation}
\tr_\t(J T)=\Str_\t(\hat \o_\C\otimes \frac{J^\cF}{\sqrt \ep} T)
\end{equation}
where $\Str_\t$ denotes the de Rham supertrace defined in \eqref{strdR}. 
\end{rem}
Let 
$$
p_\tau(t):=\Phi \tr_\tau (J e^{-\mathbb A^2_t})
$$
then the following transgression formula holds
\begin{equation}
\label{McKt}\frac{d}{dt}p_\tau(t)=d \eta_\t(t)
\end{equation}
where 
$$
\eta_\t(t)= (2\pi i)^{-\frac{1}{2}}\frac{1}{2t}\Phi \tr_\t \left(J [N-\frac{n}{2}, \mathbb X_t]e^{-\mathbb A^2_t}\right)
$$
\begin{rem} 
\label{dterm.rem}
As usual $\eta_\t(t)$ is the $d\vartheta$-term in $\tr_\t(J e^{\breve {\mathbb X}^2})$ with the construction of \eqref{X-hat}.
\end{rem}
\begin{rem} Grading disambiguation.
The operator $\mathbb X$ plays a fundamental role in the following because, as remarked in \ref{remX}, it does not contain transversal derivatives. When $\dim Z=2l$, $\mathbb X$ is an even element in $\O(B, \mathcal{W})^{dua}$.  
We use it as sort of fibrewise square root of the curvature $\mathbb A^2$. We have to be precise because
\emph{the relation between $\mathbb X$ and   
$\mathbb{A}$ depends on the grading we consider}.
On $\Omega(B, \mathcal W)^{dR}$ we have
\begin{equation}\label{AXrel}
-\mathbb{A}^2=\mathbb{X}^2  \;\text{ in  } \;\Omega(B, \mathcal W)^{dR}\ .
\end{equation}
Because $\mathbb A$ is odd both for the de Rham and the duality grading, \emph{the expression $\mathbb A^2$ is the same in the two graded algebras}. On the other hand,    
the meaning of $\mathbb X^2$ depends on the grading of $\cW$ and the resulting graded algebra structure we consider. 

It is important to stress that the only object we need is $\mathbb A^2$, in both gradings being the curvature of the superconnection. 
In $\Omega(B, \mathcal W)^{dua}$ we still have  $-\mathbb{A}^2=\mathbb X\cdot_{dR} \mathbb X$. Then the advantageous equality \eqref{AXrel} can and will be used in every grading with the small abuse of notation that we write $\mathbb X^2$ but we always mean the square of $\mathbb X$ in the de Rham grading. 

%Note also that we have $\mathbb{X}e^{-\mathbb{A}^2}\neq e^{-\mathbb{A}^2}\mathbb{X}$ in $\Omega(B, \mathcal W)^{sign}$, in accordance for example with the fact that the integrand for the eta form in \eqref{eta-sec} shall not vanish.
\end{rem}

\subsubsection{Transgression formulas, odd dimensional fibres}
\label{tra-odd}
If $\dim Z=2l+1$, assume $\nabla^\cF J^\cF=0$ and consider the family of \emph{odd signature operators}  defined as $D^{sign}_{odd}=\frac{1}{2}(d^ZJ+J d^Z)$. 
Because the operator commutes with  the chirality involution~$J$, one needs here the formalism of $\Cli$-superconnection \cite[sec. 5]{Q} and \cite[II.f]{BF}. 

The representative of the odd Chern character is
$
\trt (e^{-\mathbb{A}^2_t})^{odd} \in \O^{odd}(B)
$,
where $\mathbb{A}_t
$ is as in \eqref{AX-resc1}, \cite[sec. 5]{Q} and \cite[II.f]{BF}. 
The transgression formula here is
\begin{equation}
\label{trasgr-odd}
\frac{d}{dt}\trt{}^{odd} e^{-\mathbb{A}_T^2}=-d\; \int_0^T\trt \left(\frac{d\mathbb{A}_t}{dt}e^{-\mathbb{A}_t^2}\right)^{even}dt\ .
\end{equation}

\subsubsection{Duality structure on $L^2$-cohomology}
The bundle $H_{L^2}^\bullet (E/B;\cF)$ defined in \ref{L2coh}, with flat
Gau\ss--Manin connection~$\nabla^H$ can be given the structure of a flat
duality bundle of $\cA$-modules by means of the $\ep\ep_n$-symmetric bilinear
form
$$
Q^H_b([\phi\otimes z], [\psi\otimes \zeta])=\int_{Z_b}\phi\w \psi \cdot Q^{\mathcal{F}}(z,\zeta)\ .
$$
Recall the isomorphism \eqref{isoHodge2}
$$
H^\bullet_{L^2}(E/B;\cF)\cong\ker(d^Z+d^{Z,*})
  \subset\Omega^\bullet_{L^2}(E/B;\cF)
$$
and that $P$ denotes the projection onto the fibrewise kernel
$\ker(d^Z+d^{Z,*})$ which, by Proposition \ref{SmoothProp}, is smooth. Under
the identification above, the connection $\nabla^H$ corresponds to the
connection $P\nabla^\cW P$ on the bundle $\ker (d^Z+ d^{Z, *})$, see
\cite[Proposition 2.6]{BL}.

% $J^\cW$ induces an operator on the bundle $H$, %with $(J^H)^2=\ep\ep_n$ 
%and one can consider the characteristic class 
%%$
%%p(\nabla^H, J^H)%=\ch_\t(P^+\nabla^{H}P^+)-\ch_\t(P^-\nabla^{H}P^-)
%%$
%of the flat duality bundle $  H^\bullet_{L^2}(E/B;\cF)$ as defined in \eqref{class.p}.

Until the end of this section let $\dim Z=2l$, consider on complex forms the involution $J=\o_\C\otimes \frac{J^\cF}{\sqrt \ep}$ as above. 
Let $J^H$ be the involution induced on $H^\bullet_{L^2}(E/B;\cF)$
corresponding to $PJ^HP=PJ^H$ as $J$ commutes with $P$, and define
$P^\pm:=\frac{1\pm J^H}{2}$, $H^\pm:=P^\pm H$, $\nabla^{H,\pm}:=P^{\pm}\nabla^H
P^\pm$.

If $\nabla^\cF J^\cF=0$ then $H^+\oplus H^-$ is the so called \emph{index
  bundle} of the twisted signature operator as defined by Benameur--Heitsch
\cite[10]{BH2} and
 $p(\nabla^H, J^H)=\ch_\t(\nabla^{H,+})-\ch_\t(\nabla^{H,-})$ is its $\tau$-Chern character.

For general $J$, by Lemma \ref{lemdE*} and because $J^2=1$,
the adjoint of  $\nabla^H$ with respect to $g^H$ is given by
\begin{equation}
\label{omegaH}
\nabla^{H, *}=\nabla^{H}+J^H[\nabla^H, J^H]\ .
\end{equation}
The characteristic class $p(\nabla^H, J^H)$ can be computed as follows.
\begin{lem}\label{comput.p}
$$p_\tau(\nabla^H, J^H)=\tr_\t\left(J^H Pe^{R_0P}\right)
$$
where $R_0=\frac{1}{2} (\nabla^{\cW,*}-\nabla^{\cW})=\frac{1}{2}\o$ as defined in \eqref{omega}.
\end{lem}
\begin{proof}
 A simple computation shows that 
\begin{multline*}
\curv(P^{+}\nabla^H P^+\oplus P^{-}\nabla^H P^-)=(\nabla^HP^+)^2\\=[\nabla^H, P^+]^2
=\left[\nabla^H, \frac{1+J^H}{2}\right]^2=\frac{1}{4}[\nabla^H, J^H]^2
\end{multline*}
where we have used that the commutator $[\nabla^H,P^+]$ is multiplication by $\nabla^H P^+$ \cite{Ka}.
Then 
\begin{multline*}
p_\tau(\nabla^H, J^H)=\tr_\t\left(J^H e^{-\frac{1}{4}[\nabla^H, J^H]^2}\right)\\
=\tr_\t\left(J^H \sum_r \frac{1}{4^r}\left(-[\nabla^H, J^H]^2\right)^r\right)=\tr_\t\left(J^H \sum_r \frac{1}{r!4^r}\left( J^H[\nabla^H, J^H]^{2r}\right)\right)\\
=\tr_\t\left(J^H\sum_r \frac{1}{r!}P(R_0P)^{2r}\right)=\tr_\t\left(J^H Pe^{R_0P}\right)
\end{multline*}
where we have used that $(-1)^r[\nabla^H, J^H]^{2r}=\left(J^H[\nabla^H,
  J^H]^{2r}\right)$, and that
\begin{multline*}
PR_0P=R_0= \frac{1}{2}\left(P(J\nabla{^\cW}J-\nabla) P\right)
\stackrel{[P,J]=0}{=} \frac{1}{2}\left( PJP\nabla^{\cW}PJP -P\nabla^{\cW} P\right)\\
\frac{1}{2}\left( J^H\nabla^H J^H-\nabla^H\right)=\frac{1}{2}
\left(\nabla^{H,*}-\nabla^H\right )
\stackrel{~\eqref{omegaH}}{=} \frac{1}{2} J^H[\nabla^H,J^H].
\end{multline*}

\end{proof}

\section{The heat kernel for large times}
\label{heat-kernel}
\no In this section we prove the main theorems about the asymptotic behavior of the families $\mathbb{X}_te^{\mathbb{X}_t^2}$ and $e^{\mathbb{X}_t^2}$ as $t\to \infty$. Recall that $P=(P_b)_{b\in B}$ is the family of projections onto $\ker (d^Z+d^{Z,*})$ defined in Proposition \ref{SmoothProp}.
\begin{theorem}\label{HeatThm2}
  For~$k\in\{0,1,2\}$, 
  we have
  \begin{equation*}
    \lim_{t\to\infty}\mathbb{ X}_t^k\,e^{\mathbb{ X}_t^2}
    =P(R_0P)^k\,e^{(R_0P)^2}
    \in\O^\bullet(B,\End{}_\mathcal{A}\O^\bullet_{L^2}(E/B,\mathcal{F}))\;.
  \end{equation*}
  with respect to the $\tau$-norm.
\end{theorem}

Let $m_B=\dim B$. We denote the standard $n$-simplex by
$$
\Delta^n=\left\{(s_0,\dots,s_n)\in[0,1]^{n+1}| s_0+\dots +s_n=1\right\}
$$
and the standard volume form on $\Delta^n$ by $d^n(s_0,\dots,s_n)$,
so that $\Delta^n$ has total volume $\frac{1}{n!}$.

We split $\mathbb{X}_t$ as
\begin{equation}
\label{split}
\mathbb{X}_t=\sqrt t D+R_t
\end{equation}
 where $D=\frac{1}{2}(d^{Z,*}-d^Z)$ is a family of skew-adjoint elliptic first order differential operators along the fibres, and the remainder $R_t$ has coefficients in $\Lambda^{>0}(T^*B)$
\begin{equation}
R_t=R_0+t^{-\frac{1}{2}}R_{1} 
\end{equation}
with $R_0, R_1$ independent of $t$.
 From equation \eqref{AX},
$\mathbb{X}^2=tD^2+R_t\sqrt t D+\sqrt tDR_t+R^2_t$,
where the products are always  in $\O(B,\mathcal W)^{dR}$.

From \eqref{split}, and by Duhamel's principle 
\begin{multline}\label{Duhamel}
e^{\mathbb{ X}_t^2}=\sum_{n=0}^{m_B}\int_{\Delta^n}e^{s_0 tD^2}(\sqrt t R_tD+\sqrt t DR_t+R_t^2)e^{s_1tD^2}\dots\\
\dots (\sqrt t R_tD+\sqrt t DR_t+R_t^2)e^{s_ntD^2}d^n(s_0,\dots,s_n)\;.
\end{multline}

The strategy of the proof will be the following: we will decompose the standard simplex $\Delta^k$ into regions where certain simplex coordinates $s_i$ are smaller than a given $\bar s(t)$, and the remaining are larger.  Then we integrate over the small simplex coordinates before considering the limit as $t\to\infty$, and we make an opportune choice of $\bar s(t)$.  In this procedure, the  
heat operator $e^{\mathbb X^2}$ will split into a sum of various terms: the estimates of the resulting functions of $t D^2$ tell us which terms contribute at large time, and analyzing their combinatorics we compute its limit as $t\to\infty$. 
\subsection{Large time asymptotic: some estimates}
\label{estimates}
Let $0<\bar{s}(t)<1$ be a decreasing function of $t$, going to zero as $t\to \infty$.
We will fix~$\bar s$ in Lemma~\ref{barst} below.
Choose $T$ such that $\bar{s}(T)<\frac{1}{m_B+1}$.
\begin{lem}\label{HeatLem1}For~$c\ge 0$,
  there exists a constant~$C$ such that for all~$s>0$, $t>T$,
  \begin{align}
    \label{estOPc}
    \n{(\sqrt t\,D)^c e^{stD^2}}_{\op}&\le C\,s^{-\tfrac c2}\ ,
        &&\text{for }c\geq 0
    \\
    \label{estTR0}
    \n{e^{stD^2}-P}_\tau&=\theta(st)\;,  &&\text{for }\theta \text{ of \ref{NStheta}}\\
%    \label{estTRtot}
%    \n{e^{stD^2}}_\tau&\le C\;,\\
    \label{estTRc}
    \n{(\sqrt t\,D)^c e^{stD^2}}_{\tau}
    &\le C\,s^{-\tfrac c2}\theta\Bigl(\frac{st}2\Bigr)
    &&\text{for }c\geq 1\;.
  \end{align}
\end{lem}

\begin{proof} The first two estimates are immediate. For the last one write
\begin{equation*}
  \n{(\sqrt t\,D)^c e^{stD^2}}_{\tau}
 % =\n{(\sqrt t\,D^2)^c e^{\frac{stD^2}{2}}\cdot \left(e^{\frac{stD^2}{2} }-P\right)}_\tau\leq \\
  \leq\n{(\sqrt t\,D)^c e^{\frac{stD^2}{2}}}_{\op}\cdot \n{e^{\frac{stD^2}{2} }-P}_\tau\leq C\,s^{-\tfrac c2} \theta\Bigl(\frac{st}2\Bigr)\;.\qedhere
\end{equation*}
\end{proof}

When the parameter $s$ is small and $c=0,1$,
we get better estimates by integrating over~$s$.
The case~$c=2$ is more complicated and will be treated later.

\begin{lem}\label{HeatLem2}
Let $c\in \{0,1\}$. There exists a constant~$C$ such that for all~$t>T$, we have
\begin{equation}
\label{OPint0}
\n{\int_0^{\bar{s}}(\sqrt t\,D)^ce^{stD^2}ds}_{\op}\leq \bar{s}^{1-\tfrac c2}\cdot C
\end{equation}
\end{lem}
\begin{proof}
This follows by integrating~\eqref{estOPc}.
\end{proof}

\begin{rem}\label{unif.estimates} The estimates of Lemma \ref{HeatLem1} and  \ref{HeatLem2} can be made uniformly on compact subsets of $B$, as follows from the discussion after equation \eqref{NStheta}.
\end{rem}

\subsection{Choices for $\bar s(t)$}

\begin{lem}
\label{barst}  
  Recall $\theta(t)=\tr(e^{tD^2}-P)$, defined in \eqref{NStheta}.  
  \begin{enumerate}
\item \label{barS} There exists a choice of a monotone decaying function $\bar s=\bar s(t)$ such that 
$$\lim_{t\to \infty}\;\theta\left(\frac{t\bar{s}(t)}{2}\right)\cdot \left(\frac{1}{\bar{s}(t)}\right)^{\frac{m_B}{2}}=0 \ .$$
\item  \label{barS+}If  there exists $\a>0$ such that 
$\theta(t)=\mathcal{O}(t^{-\a})$ (that is, if $D$ has positive Novikov--Shubin invariants), then there exist a function $\bar s(t)$ and $\ep>0 $ such that
$$
\theta(t\bar{s}(t))\cdot \left(\frac{1}{\bar{s}(t)}\right)^\frac{m_B}{2}\leq t^{-\ep}, \text{ as } t\to\infty\ .
$$
\item  \label{barSdet} If $\ds\int_1^\infty \theta(t) \frac{dt}{t}<\infty$
  (that is, if $D$ is of determinant class, Definition \ref{defdetcl}), then  for each $ d\geq 0$
 there exists a choice of a monotone decaying function $\bar s=\bar s(t)$ such that 
 $$
   \int_1^\infty \theta(t\bar s(t))\left(\frac{1}{\bar s(t)}\right)^d \frac{dt}{t}<\infty\ .
 $$

\end{enumerate}

\end{lem}
\begin{proof}
To prove \eqref{barS}, let $\psi $ be the inverse function of $\theta$. The
function $\ep\mapsto 2 \ep^{-1}\psi(\ep^{\frac{m_B}{2}+1})$ is monotone
decreasing (as product of decreasing factors) and therefore has an inverse
which we take to be the function $\bar s(t)$, that is $\ep=\bar s(t)$. We have
\begin{equation}
\label{ }
t=\bar{s}^{-1}(\ep):=2\ep^{-1}\cdot\psi\left(\ep^{ \frac{m_B}{2}+1}\right)
\end{equation}
As $\lim_{\ep\rightarrow 0}\overline{s}^{-1}(\ep)=\infty$ if follows that
$\lim_{t\to \infty }\bar{s}(t)=0$. Moreover, by
construction $$\theta\left(\frac{t\bar{s}(t)}{2}\right)\cdot
\left(\frac{1}{\bar{s}(t)}\right)^{\frac{m_B}{2}}=\bar s(t)\xrightarrow{t\to\infty} 0.
$$
To prove \eqref{barS+}, choose 
 $\beta $ such that $\left(1+\frac{\a}{\frac{m_B}{2}+1}\right)^{-1}<\b<1$, hence $1-\beta<\frac{\a\b}{\frac{m_B}{2}+1}$.   Define then $\bar{s}(t)=t^{\beta-1}$. It follows 
 
%\begin{equation}
%\theta(t\bar{s}(t))\cdot \left(\frac{1}{\bar{s}(t)}\right)^\frac{m_B}{2}\leq t^{-\a\b}(t^{1-\b})^{m_B}\leq t^{-\a\b}t^{\frac{2\a\b}{m_B+2}m_B}=t^{-\frac{m+1}{m+2}\a\b}=t^{-\ep}\;\;, \epsilon>0\ .
%\end{equation}
\begin{equation}
\theta(t\bar{s}(t))\cdot \left(\frac{1}{\bar{s}(t)}\right)^\frac{m_B}{2}\leq t^{-\a\b}(t^{1-\b})^{m_B}\leq t^{-\frac{m+1}{m+2}\a\b}=t^{-\ep}\;\;, \epsilon>0\ .
\end{equation}
To prove \eqref{barSdet}, we make the following construction. Choose $T_1<T_2<\dots$ such that for each $k$
$$
\int_{\frac{T_k}{2^k}}^\infty\theta (t)\frac{dt}{t}< 2^{-k(d+1)}\ .
$$
Then $\ds\int_{\frac{T_k}{2^k}}^\infty2^{dk}\cdot\theta (t)\frac{dt}{t}<2^{-k}$. Define $\bar s(t):= 2^{-k}$ for $T_k\leq t\leq T_{k+1}$.
Now 
  \begin{multline*}  \begin{aligned}
  \int_1^\infty \theta(t\bar s(t))\bar s(t)^{-d}\frac{dt}{t}=&\sum_k\int_{T_k}^{T_{k+1}}\theta(2^{-k}t)2^{dk}\frac{dt}{t}\\
& \leq \sum_{k=1}^\infty \int_{\frac{T_k}{2^k}}^\infty \theta(r)\cdot 2^{dk}\frac{dr}{r}<\sum_1^\infty 2^{-k}<\infty\ .
\end{aligned}
  \end{multline*}
 \end{proof}

Cases (2) and (3) of the lemma above will only be used in Sections
\ref{refined}. Note that in the following section we do not make any specific assumption on $\theta(t)$.

\subsection{Splitting Duhamel's formula}\label{spl-Duham}
For $n\leq m_B$, split $\Delta^n=\ds\bigcup_{I\subsetneq\{0,\dots n\}}\Delta^n_{\bar{s},I}$,
where 
\begin{equation*}
\Delta^n_{\bar{s},I}=\{(s_0,\dots , s_n)\mid s_i\leq \bar{s} \;\text{ if and only if }\;i\in I \}\;.
\end{equation*}
As $T>0$ is chosen such that $\bar s(T)<\frac1{m_B+1}$, we have that for all~$t>T$ and all~$(s_0, \dots ,s_n)\in \Delta^n$ there is at least one variable $s_i$ such that $s_i>\tilde s(t)$, so that~$\Delta_{\bar s(t),\{0,\dots,n\}}=\emptyset$.

For fixed~$n\ge 0$ and~$I\subset\{0,\dots,n\}$,
we will regard each of the $3^n$ terms in~\eqref{Duhamel} of the form 
\begin{equation}\label{term}
\int_{\Delta^n_{\bar{s},I}}e^{s_0 tD^2}S_1e^{s_1tD^2}\dots S_ne^{s_ntD^2}d^n(s_0, \dots ,s_n)
\end{equation}
separately,
where $S_i\in \{\sqrt t DR_t, \sqrt t R_tD, R_t^2\}$. 
We group $\sqrt t D$ and its neighbors which are functions of $D$ so that we have factors of the form 
$$
(\sqrt tD)^{c_i}e^{s_itD^2}\text{ with }c_i\in\{0,1,2\}\qquad\text{and}\quad
R_t^a\text{ with }a\in\{1,2\}\;.
$$
We write a single term as 
\begin{multline}\label{int-spez}
K(t,n,I;c_0,\dots,c_n;a_1,\dots,a_n)
=\int_{\Delta_{\bar s(t),I}^n}
(\sqrt t\,D)^{c_0}e^{s_0 tD^2}R_t^{a_1}(\sqrt t\,D)^{c_1}e^{s_1tD^2}R_t^{a_2}\dots \\
\dots (\sqrt t\,D)^{c_n}e^{s_ntD^2}\,d^n(s_0,\dots,s_n)
\end{multline}
with~$c_i\ge 0$ and~$a_i>0$ for all~$i$.
Note that by~\eqref{Duhamel} and the above,
we can write~$e^{\hat{\mathbb X}_t^2}$ as sum of terms
of the form~$K(n,I;c_0,\dots,c_n;a_1,\dots,a_n)$;
however,
not all possible combinations of~$c_i$ and~$a_i$ occur in this sum.
With this notation,
Duhamel's formula~\eqref{Duhamel} now becomes
\begin{equation}\label{SplitDuhamel}
  e^{\hat{\mathbb X}_t^2}
  =\sum_n\sum_I\sum_{
    \begin{smallmatrix}
      \phantom{a_0+}c_0+a_1\le 2\le c_0+a_1+c_1\\
      c_0+\dots+a_2\le 4\le c_0+\dots+c_2\\
      \vdots\\
      c_0+a_1+\dots+a_n+c_n=2n
    \end{smallmatrix}}
  K(t,n,I;c_0,\dots,c_n;a_1,\dots,a_n)\;.
\end{equation}
Using the estimates of the Lemmas~\ref{HeatLem1} and~\ref{HeatLem2},
we show that some of the terms above vanish for $t\to\infty$ in $\tau$-norm.

\begin{prop}\label{EasyTermsProp}
  As~$t\to\infty$,
  we have the following asymptotics with respect to the $\tau$-norm.
  \begin{enumerate}
  \item\label{Easy1}
    If~$I=\emptyset$, then
    $$\lim_{t\to\infty}K(t,n,I;c_0,\dots,c_n;a_1,\dots,a_n)=
    \begin{cases}
%  corrected R_t to R_0 in the above
     \tfrac1{n!}\,P\,R_0^{a_1}\,PR_0^{a_2}P\cdots P&\text{if~$c_0=\dots=c_n=0$,}\\  
      0	&\text{otherwise.}
    \end{cases}$$
  \item\label{Easy2}
    If~$I\ne \emptyset$ and~$c_i\in\{0,1\}$ for all~$i\in I$,
    then
	$$\lim_{t\to\infty}K(t,n,I;c_0,\dots,c_n;a_1,\dots,a_n)=0\;.$$
  \end{enumerate}
  Moreover in each of the cases considered above, for $t$ sufficiently large
  \begin{multline*}
  \left|K(t,n,I;c_0,\dots,c_n;a_1,\dots,a_n)-\lim_{t\to\infty}K(t,n,I;c_0,\dots,c_n;a_1,\dots,a_n)\right|_\tau\leq \\
  \leq C\left(\bar s(t)^{-\tfrac{m_B}2}
	\,\theta\biggl(\frac{\bar s(t)\,t}2\biggr)+ \n{P}_\tau\bar s(t)^{\frac{1}{2}}\right).
\end{multline*}
\end{prop}

\begin{proof}
  For~\eqref{Easy1},
  we note first that~$a_1+\dots+a_n\le m_B$.
  Because each~$a_i\ge 1$, this implies~$n\le m_B$ and
  \begin{equation}\label{CnBound}
    c_0+\dots+c_n=2n-a_1-\dots-a_n\le n\le m_B\;.
  \end{equation}

  Assume first that~$c_i\ne 0$ for some~$0\le i\le n$,
  for simplicity~$c_0\ne 0$.
  Because~$I=\emptyset$, we have~$s_j\ge\bar s(t)$ for all~$j$
  if~$(s_0,\dots,s_n)\in\Delta^n_{\bar s(t),\emptyset}$.
  Using Lemma~\ref{HeatLem1}, we find that
  \begin{multline}\label{rest1}
    \n{(\sqrt t\,D)^{c_0}e^{s_0 tD^2}R_t^{a_1}(\sqrt t\,D)^{c_1}e^{s_1tD^2}
      \cdots(\sqrt t\,D)^{c_n}e^{s_ntD^2}}_\tau\\
    \begin{aligned}
      &\le\n{(\sqrt t\,D)^{c_0}e^{s_0 tD^2}}_\tau\cdot\n{R_t^{a_1}}_{\op}
	\cdot\n{(\sqrt t\,D)^{c_1}e^{s_1 tD^2}}_{\op}\cdots
	\n{(\sqrt t\,D)^{c_n}e^{s_n tD^2}}_{\op}\\
      &\le C_1\,s_0^{-\tfrac{c_0}2}\,\theta\biggl(\frac{s_0t}2\biggr)
	\,\n{R_t^{a_1}}_{\op}
	\cdot C_2\,s_1^{-\tfrac{c_1}2}\cdots C_3\,s_n^{-\tfrac{c_n}2}
      \le C\,\bar s(t)^{-\tfrac{m_B}2}
	\,\theta\biggl(\frac{\bar s(t)\,t}2\biggr)
    \end{aligned}
  \end{multline}
  for some constant~$C$. % that changes from term to term.
  
  Chose ~$\bar s(t)$ as in Lemma~\ref{barst} (1).
  Then
	$$\lim_{t\to\infty}\n{(\sqrt t\,D)^{c_0}e^{s_0 tD^2}R_t^{a_1}
		(\sqrt t\,D)^{c_1}e^{s_1tD^2}\cdots
    		(\sqrt t\,D)^{c_n}e^{s_ntD^2}}_\tau
	=0$$
  uniformly on~$\Delta_{\bar s(t),\emptyset}^n$.
  Hence in this case,
	$$\lim_{t\to\infty}
	\n{K(t,n,\emptyset;c_0,\dots,c_n;a_1,\dots,a_n)}_\tau=0\;.$$

  If~$I=\emptyset$ and~$c_0=\dots=c_n=0$,
  we compute
  \begin{multline*}
    \n{\bigl(e^{s_0 tD^2}-P\bigr)\,R_t^{a_1}\,e^{s_1tD^2}
      \cdots e^{s_ntD^2}}_\tau\\
    \le\n{e^{s_0 tD^2}-P}_\tau\cdot\n{R_t^{a_1}}_{\op}
	\cdot\n{e^{s_1 tD^2}}_{\op}\cdots
	\n{(\sqrt t\,D)^{c_n}e^{s_n tD^2}}_{\op}
    \le C\,\theta\biggl(\frac{\bar s(t)\,t}2\biggr)\;,
  \end{multline*}
  which tends to~$0$ as~$t\to\infty$.
  By repeating this computation successively for~$s_1$, \dots, $s_n$,
  we find that
   \begin{multline*}
    \lim_{t\to\infty}e^{s_0 tD^2}\,R_0^{a_1}\,e^{s_1tD^2}\cdots e^{s_ntD^2}
    =P\,R_0^{a_1}\,\lim_{t\to\infty}e^{s_1tD^2}\,R_t^{a_2}\,e^{s_2tD^2}
	\cdots e^{s_ntD^2}\\
    =\dots=P\,R_0^{a_1}\,P\cdots P
  \end{multline*}
% corrected R_t to R_0 in the above
  uniformly
  on~$\Delta_{\bar s(t),\emptyset}^n$ with respect to the $\tau$-norm.
  Because
	$$\lim_{t\to\infty}\vol(\Delta_{\bar s(t),\emptyset}^n)
	=\vol(\Delta^n)=\frac1{n!}\;,$$
  integrating over~$\Delta_{\bar s(t),\emptyset}^n$ proves the remaining
  case in~\eqref{Easy1}.

  Now assume that~$I\ne\emptyset$ and
  put~$I:=\{i_1, \dots i_r\}$
  and $\{0, \dots , n\}\setminus I=:\{j_0,\dots , j_{n-r}\}\ne\emptyset$
  because by our choice of~$T$, we have~$r\le n$.
  As in~\eqref{Easy2},
  we assume~$c_{i_1}$, \dots, $c_{i_r}\in\{0,1\}$.
  We rewrite~\eqref{int-spez} as
  \begin{multline}\label{integr}
    K(t,n,I;c_0,\dots,c_n;a_1,\dots,a_n)
    =\underbrace{\int_0^{\bar{s}(t)}\dots\int_0^{\bar{s}(t)}}_{r\text{ times}}
    \int_{\{(s_{j_0},\dots,s_{j_{n-r}})
	\mid(s_0,\dots,s_n)\in\Delta^n_{\bar{s}(t),I}\}}\\
    (\sqrt t\,D)^{c_0}e^{s_0 tD^2}R_t^{a_1}(\sqrt t\,D)^{c_1}e^{s_1tD^2}
	\cdots(\sqrt t\,D)^{c_n}e^{s_ntD^2}\,d^{n-r}(s_{j_0},\dots,s_{j_{n-r}})
        \,ds_{i_r}\cdots ds_{i_1}\;.
  \end{multline}
  To estimate the $\tau$-norm,
  we take the $\tau$-norm of~$(\sqrt t\,D)^{c_{j_0}}e^{s_{j_0}tD^2}$
  and the operator norms of the remaining factors.
  
  Assume first that there exists $j\in \{0, \dots , n\}\setminus I$ such that $c_j\ge 1$, say $c_{j_0}\ge 1$. Then by \ref{estTRc},
  \begin{multline}\label{rest3}
    \n{K(t,n,I;c_0,\dots,c_n;a_1,\dots,a_n)}_\tau\\
    \begin{aligned}
      &\le\int_0^{\bar{s}(t)}\dots\int_0^{\bar{s}(t)}
	\int_{\{(s_{j_0},\dots,s_{j_{n-r}})
	\mid(s_0,\dots,s_n)\in\Delta^n_{\bar{s}(t),I}\}}\\
      &\kern4em
	C\,s_0^{-\tfrac{c_0}2}\cdots s_n^{-\tfrac{c_n}2}
	\,\theta\biggl(\frac{s_{j_0}t}2\biggr)
	\,d^{n-r}(s_{j_0},\dots,s_{j_{n-r}})
	\,ds_{i_r}\cdots ds_{i_1}\\
      &\stackrel{(\ref{CnBound})}{\le}\int_0^{\bar{s}(t)}\dots\int_0^{\bar{s}(t)}
	\vol{}^{n-r}\{(s_{j_0},\dots,s_{j_{n-r}})
	\mid(s_0,\dots,s_n)\in\Delta^n_{\bar{s}(t),I}\}\\
      &\kern4em
        C\,\bar s(t)^{-\frac{m_B}{2}}\,s_{i_1}^{-\tfrac{c_{i_1}}2}
	\cdots s_{i_r}^{-\tfrac{c_{i_r}}2}
	\,\theta\biggl(\frac{\bar s(t)\,t}2\biggr)
	\,ds_{i_r}\cdots ds_{i_1}\\
      &\le C\,\bar s(t)^{-\frac{m_B}{2}}\,\theta\biggl(\frac{\bar s(t)\,t}2\biggr)\;.
    \end{aligned}
  \end{multline}
  Again,
  this tends to~$0$ as~$t\to\infty$ by our choice of~$\bar s(t)$.
  
  If $c_{j_0}=\cdots= c_{j_p}=0$, replace $e^{s_{j_0}tD^2}$ by
  $(e^{s_{j_0}tD^2}-P)+P$ and estimate it by taking its $\tau$-norm  and the
  operator norm of the remaining factor, which does not contribute with
  negative powers of $\bar s(t)$, since $c_{j_0}=\cdots= c_{j_p}=0$:
    \begin{multline*}
    \n{K(t,n,I;c_0,\dots,c_n;a_1,\dots,a_n)}_\tau\\
    \begin{aligned}
      & \le \int_0^{\bar{s}(t)}\dots\int_0^{\bar{s}(t)}
    \int_{\{(s_{j_0},\dots,s_{j_{n-r}})
	\mid(s_0,\dots,s_n)\in\Delta^n_{\bar{s}(t),I}\}} \left(\n{e^{s_{j_0}tD^2}-P}_\t+\n{P}_\t\right) \cdot \n{R_t^{a_1}}_{op} \cdots \\
	&\kern4em \cdots \n{e^{s_n t D^2}}_{op} d^{n-r}(s_{j_0},\dots,s_{j_{n-r}}) ds_{i_r}\cdots ds_{i_1}
	\\
      &\le\int_0^{\bar{s}(t)}\dots\int_0^{\bar{s}(t)}
	\int_{\{(s_{j_0},\dots,s_{j_{n-r}})
	\mid(s_0,\dots,s_n)\in\Delta^n_{\bar{s}(t),I}\}}\\
      &\kern4em
	C\,s_{i_1}^{-\tfrac{c_{i_1}}2}
	\cdots s_{i_r}^{-\tfrac{c_{i_r}}2} \,\left(\theta\biggl(\frac{s_{j_0}t}2\biggr)+\n{P}_\t\right)
	\,d^{n-r}(s_{j_0},\dots,s_{j_{n-r}})
	\,ds_{i_r}\cdots ds_{i_1}\\
	&\le \bar s(t)^{r- \frac{c_{i_1}}{2}-\dots -\frac{c_{i_r}}{2} } \left(\theta\biggl(\frac{\bar s(t)t}2\biggr)+\n{P}_\t\right) \vol{}^{n-r}\{(s_{j_0},\dots,s_{j_{n-r}})	\mid(s_0,\dots,s_n)\in\Delta^n_{\bar{s}(t),I}\}\ .
    \end{aligned}
  \end{multline*}
  which goes to $0$ as $t\to\infty$, because $r-\frac{c_{i_1}}{2}-\dots -\frac{c_{i_r}}{2}>0$.
  
\end{proof}

\subsection{Integration by parts}\label{int.part}
To estimate the $\t$-norm of \eqref{int-spez} if~$c_i=2$ for some~$i\in I$
using the lemmas in Section \ref{estimates},
we proceed to eliminate all terms of the form $tD^2 e^{s_{i_a}t D^2}$,
$i_a\in I$ by integration by parts.

As a preparation, let~$g\colon[0,\infty)^{n-r+1}\to\C$ be a function of class $C^1$,
let~$q=n-r$ and assume that~$0<\sigma<s_0$
and~$c>(q+1)\bar s+\sigma$.
We first want to compute the derivative of the integral of~$g$
over the interior part of the simplex where all variables are at least~$\bar s$,
with respect to the size~$c-\sigma$ of the simplex.
We find
\begin{multline*}
  -\frac\partial{\partial\sigma}
  	\int_{\{\,(x_0,\dots,x_q)\in(c-\sigma)\Delta^q\mid x_0,
	  \dots,x_q\ge\bar s\,\}}g(x_0,\dots,x_q)\,d^q(x_0,\dots,x_q)\\
    =-\frac\partial{\partial\sigma}\int_{\bar s}^{c-\sigma-q\bar s}
        \int_{\bar s}^{c-\sigma-(q-1)\bar s-x_0}
	\cdots\int_{\bar s}^{c-\sigma-\bar s-x_0-\dots-x_{q-2}}\\
	g(x_0,\dots,x_{q-1},c-\sigma-x_0-\dots-x_{q-1})
	\,dx_{q-1}\cdots dx_0
	\end{multline*}
\begin{multline*}
  \begin{aligned}
    &\kern6em=\biggl(\int_{\bar s}^{c-\sigma-(q-1)\bar s-x_0}
	\cdots\int_{\bar s}^{c-\sigma-\bar s-x_0-\dots-x_{q-2}}\\
    &\kern6em
	g(x_0,\dots,x_{q-1},c-\sigma-x_0-\dots-x_{q-1})
	\,dx_{q-1}\cdots dx_1\biggr)\biggr|_{x_0=c-\sigma-q\bar s}\\
    &\qquad+\dots+\int_{\bar s}^{c-\sigma-q\bar s}
	\cdots\int_{\bar s}^{c-\sigma-2\bar s-x_0-\dots-x_{q-3}}\\
    &\kern6em
	g(x_0,\dots,x_{q-1},c-\sigma-x_0-\dots-x_{q-1})\bigr|_{x_{q-1}=c-\sigma-\bar s-x_0-\dots-x_{q-2}}
	\,dx_{q-2}\cdots dx_0\\
    &\qquad+\int_{\bar s}^{c-\sigma-q\bar{s}}\int_{\bar s}^{c-\sigma-(q-1)\bar s-x_0}
	\cdots\int_{\bar s}^{c-\sigma-\bar s-x_0-\dots-x_{q-2}}\\
    &\kern6em
	\frac{\partial g}{\partial x_q}(x_0,\dots,x_{q-1},c-\sigma-x_0-\dots-x_{q-1})
	\,dx_{q-1}\cdots dx_0\;.\\
  \end{aligned}
\end{multline*}
The first~$q$ terms arise by formal differentiation of an integral with respect
to its upper limit.
The first~$q-1$ of them vanish because there remains at least one inner
integral over an interval of length~$0$.
Thus, we are left with
\begin{multline*}
  \frac\partial{\partial\sigma}
  	\int_{\{\,(x_0,\dots,x_q)\in(c-\sigma)\Delta^q\mid x_0,
	  \dots,x_q\ge\bar s\,\}}g(x_0,\dots,x_q)\,d^q(x_0,\dots,x_q)\\
  =-\int_{\{\,(x_0,\dots,x_{q-1})\in(c-\sigma-\bar s)\Delta^{q-1}\mid x_0,
	  \dots,x_{q-1}\ge\bar s\,\}}g(x_0,\dots,x_{q-1},\bar s)\,d^{q-1}(x_0,\dots,x_{q-1})\\
  -\int_{\{\,(x_0,\dots,x_q)\in(c-\sigma)\Delta^q\mid x_0,
	  \dots,x_q\ge\bar s\,\}}
	\frac{\partial g}{\partial x_q}(x_0,\dots,x_q)
	\,d^q(x_0,\dots,x_q)
\end{multline*}
if~$q\ge 1$,
and since~$(c-\sigma)\Delta^0=\{c-\sigma\}$, we have
\begin{equation*}
  -\frac\partial{\partial\sigma}g(c-\sigma)
  =\frac{\partial g}{\partial x_0}(c-\sigma)
\end{equation*}
if~$q=0$.
The last simplex variable~$x_q$ plays a special role in this computation,
so we will call it the ``target variable'' later on.
By symmetry of integration,
we may choose any of the simplex variables to be our target variable.

Now assume that the term~$tD^2e^{s_{i_a}tD^2}$ occurs somewhere in
one of the factors~\ref{int-spez} with~$i_a\in I$.
At this point, we integrate only over~$s_{i_a}$ and~$s_{j_0}$, \dots,
$s_{j_q}$ and keep all other small variables~$s_{i_b}$ with~$b\ne a$ fixed.
Recall that there exists at least one~$j_0\notin I$.
We choose~$s_{j_0}$ as target variable.
By the above, from the equality
\begin{multline*}
\int_{\{\,(s_{j_0},\dots,s_{j_q})\mid
	s_{j_0},\dots,s_{j_q}\ge\bar s\,,\;s_0+\dots+s_n=1\,\}}\\
	\begin{aligned}
	&\kern6em 
	\ldots R_t\,e^{s_{i_a}tD^2}\,R_t\ldots (\sqrt tD)^{c_{j_0}}e^{s_{j_0}tD^2}\ldots
	\,d^q(s_{j_0},\dots,s_{j_q})\Bigr|_{s_{i_a}=0}^{\bar s}=\\
	&=\qquad\int_0^{\bar s}\frac\partial{\partial s_{i_a}}
      \int_{\{\,(s_{j_0},\dots,s_{j_q})\mid
	s_{j_0},\dots,s_{j_q}\ge\bar s\,,\;s_0+\dots+s_n=1\,\}}\\
    &\kern6em
	\ldots R_t\,e^{s_{i_a}tD^2}\,R_t\ldots (\sqrt tD)^{c_{j_0}}e^{s_{j_0}tD^2}\ldots
	\,d^q(s_{j_0},\dots,s_{j_q})\,ds_{i_a}\\
	\end{aligned}
\end{multline*}
we obtain
\begin{multline*}
  \int_0^{\bar s}\int_{\{\,(s_{j_0},\dots,s_{j_q})\mid
	s_{j_0},\dots,s_{j_q}\ge\bar s\,,\;s_0+\dots+s_n=1\,\}}\\
  \begin{aligned}	
    &\kern6em
	\ldots R_t\,tD^2\,e^{s_{i_a}tD^2}\,R_t
	\ldots (\sqrt tD)^{c_{j_0}}e^{s_{j_0}tD^2}\ldots
	\,d^q(s_{j_0},\dots,s_{j_q})\,ds_{i_a}\\ 
        &=\int_{\{\,(s_{j_0},\dots,s_{j_q})\mid
	s_{j_0},\dots,s_{j_q}\ge\bar s\,,\;
	s_0+\dots+\widehat{s_{i_a}}+\dots+s_n=1-\bar s\,\}}\\
    &\kern6em
	\ldots R_t\,e^{\bar stD^2}\,R_t
	\ldots (\sqrt tD)^{c_{j_0}}e^{s_{j_0}tD^2}\ldots
	\,d^q(s_{j_0},\dots,s_{j_q}) \\
    &\qquad-\int_{\{\,(s_{j_0},\dots,s_{j_q})\mid
	s_{j_0},\dots,s_{j_q}\ge\bar s\,,\;
	s_0+\dots+\widehat{s_{i_a}}+\dots+s_n=1\,\}}\\
    &\kern6em
	\ldots R_t^2\ldots (\sqrt tD)^{c_{j_0}}e^{s_{j_0}tD^2}\ldots
	\,d^q(s_{j_0},\dots,s_{j_q})\\
    &\qquad-\int_0^{\bar s}\int_{\{\,(s_{j_1},\dots,s_{j_q})\mid
	s_{j_1},\dots,s_{j_q}\ge\bar s\,,\;
	s_0+\dots+\widehat{s_{j_0}}+\dots+s_n=1-\bar s\,\}}\\
    &\kern6em
	\ldots R_t\,e^{s_{i_a}tD^2}\,R_t
	\ldots (\sqrt tD)^{c_{j_0}}e^{\bar stD^2}\ldots
	\,d^{q-1}(s_{j_1},\dots,s_{j_q})\,ds_{i_a}\\
    &\qquad-\int_0^{\bar s}\int_{\{\,(s_{j_0},\dots,s_{j_q})\mid
	s_{j_0},\dots,s_{j_q}\ge\bar s\,,\;
	s_0+\dots+\dots+s_n=1\,\}}\\
    &\kern6em
	\ldots R_t\,e^{s_{i_a}tD^2}\,R_t
	\ldots (\sqrt tD)^{c_{j_0}+2}e^{s_{j_0}tD^2}\ldots
	\,d^q(s_{j_0},\dots,s_{j_q})\,ds_{i_a}
  \end{aligned}
\end{multline*}
if~$q>0$, and a similar expression without the third term on the right hand
side if~$q=0$.

Let us now extend our notation in~\eqref{int-spez}
to incorporate those situations where some of the~$s_i$ are ``frozen''
to~$\bar s(t)$.
If~$I$ and~$J$ are disjoint subsets of~$\{0,\dots,n\}$
with~$I=\{i_1,\dots, i_r\}$ and~$\{0,\dots,n\}\setminus(I\cup J)
=:\{k_1,\dots,k_q\}\ne\emptyset$,
we write
\begin{multline*}
  K(t,n,I,J;c_0,\dots,c_n;a_1,\dots,a_n)
  =\underbrace{\int_0^{\bar s(t)}\dots\int_0^{\bar s(t)}}_{r\text{ times}}
  \int_{\{(s_{k_0},\dots,s_{k_q})
	\mid(s_0,\dots,s_n)\in\Delta^n_{\bar{s},I}\}}\\
  (\sqrt t\,D)^{c_0}e^{s_0 tD^2}R_t^{a_1}(\sqrt t\,D)^{c_1}e^{s_1tD^2}
	\cdots (\sqrt t\,D)^{c_n}e^{s_ntD^2}
	\,d^q(s_{k_0},\dots,s_{k_q})\,ds_{i_r}\dots ds_{i_1}\;,
\end{multline*}
where~$s_j=\bar s(t)$ is ``frozen'' for all~$j\in J$.
Then our computations above become
\begin{multline}\label{PartIntResult}
  K(t,n,I\cup\{i_a\},J;\dots,\underbrace{2}_{i_a},\dots,c_{k_0},\dots;
	\dots,a_{i_a},a_{i_a+1},\dots)\\
  =\begin{cases}
     \begin{aligned}
       &K(t,n,I,J\cup\{i_a\};\dots,0,\dots,c_{k_0},\dots;
	\dots,a_{i_a},a_{i_a+1},\dots)\\
       &\quad-K(t,n-1,I,J;\dots,\dots,c_{k_0},\dots;
	\dots,a_{i_a}+a_{i_a+1},\dots)\\
       &\quad+K(t,n,I\cup\{i_a\},J\cup\{k_0\};\dots,0,\dots,c_{k_0},\dots;
	\dots,a_{i_a},a_{i_a+1},\dots)\\
       &\quad+K(t,n,I\cup\{i_a\},J;\dots,0,\dots,c_{k_0}+2,\dots;
	\dots,a_{i_a},a_{i_a+1},\dots)
     \end{aligned}
     &\text{if~$q>0$, and}\vadjust{\medskip}\\
     \begin{aligned}
       &K(t,n,I,J\cup\{i_a\};\dots,0,\dots,c_{k_0},\dots;
	\dots,a_{i_a},a_{i_a+1},\dots)\\
       &\quad-K(t,n-1,I,J;\dots,\dots,c_{k_0},\dots;
	\dots,a_{i_a}+a_{i_a+1},\dots)\\
       &\quad+K(t,n,I\cup\{i_a\},J;\dots,0,\dots,c_{k_0}+2,\dots;
	\dots,a_{i_a},a_{i_a+1},\dots)
     \end{aligned}
     &\text{if~$q=0$.}
   \end{cases}
\end{multline}

We now continue to perform partial integration,
thus eliminating all terms with~$c_i=2$ for some~$i\in I$.
The remaining terms are all
of the form~$K(t,n,I,J;c_0,\dots,c_n;a_1,\dots,a_n)$ with~$c_i\in\{0,1\}$
for~$i\in I$ and~$c_i\ge 0$ for~$i\notin I$.
During partial integration, the sum of the~$c_i$ never increases,
so we still have
	$$c_0+\dots+c_n\le m_B$$
as in~\eqref{CnBound}.
We can now prove for the resulting terms an analogue of Proposition~\ref{EasyTermsProp}.

\begin{prop}\label{OtherTermsProp}
  Assume that~$c_i\in\{0,1\}$ for all~$i\in I$,
  then with respect to the $\tau$-norm,
  we have
  \begin{multline*}
    \lim_{t\to\infty}K(t,n,I,J;c_0,\dots,c_n;a_1,\dots,a_n)\\=
    \begin{cases}
      \tfrac1{(n-|J|)!}\,P\,R_0^{a_1}\,P\cdots P
        &\text{if~$I=\emptyset$ and~$c_0=\dots=c_n=0$, and}\\
      0	&\text{otherwise.}
    \end{cases}
  \end{multline*}
   Moreover in each of the cases considered above, for $t$ sufficiently large
  \begin{multline}\label{esti}
  \left|K(t,n,I,J;c_0,\dots,c_n;a_1,\dots,a_n)-\lim_{t\to\infty}K(t,n,I,J;c_0,\dots,c_n;a_1,\dots,a_n)\right|_\tau\leq \\
  \leq C\left(\bar s(t)^{-\tfrac{m_B}2}
	\,\theta\biggl(\frac{\bar s(t)\,t}2\biggr)+ \n{P}_\tau\bar s(t)^{\frac{1}{2}}\right).
\end{multline}
\end{prop}

\begin{proof}
  This is proved precisely as Proposition~\ref{EasyTermsProp}.
  If~$I=\emptyset$ and~$c_0=\dots=c_n=0$,
  we successively replace~$e^{s_itD^2}$ by~$P$
  and use~\eqref{estTR0} of Lemma~\ref{HeatLem1}.
  Because~$s_j=\bar s(t)$ for all~$j\in J$, and $\bar s(t)\to 0$ as~$t\to\infty$,
  we are left with an integral over an $(n- |J|)$-simplex
  of volume~$\frac 1{(n-|J|)!}$.

%  If~$I\ne\emptyset$ or there exists~$c_i\ne 0$ for some~$i\in\{0,\dots,n\}$,
%  the arguments in the proof of~\ref{EasyTermsProp} show
%  that~$\n{K(t,n,I,J;c_0,\dots,c_n;a_1,\dots,a_n)}_\tau\to 0$.
%\end{proof}
 If~$I=\emptyset$ and there exists~$c_i> 0$ for some~$i\in\{0,\dots,n\}$,
  the arguments in the proof of~\ref{EasyTermsProp} show
  that~$\n{K(t,n,I,J;c_0,\dots,c_n;a_1,\dots,a_n)}_\tau\to 0$.
  
  If $I\ne \emptyset$, then we proceed again exactly as in the proof of~\ref{EasyTermsProp}, because the frozen variables play the same role as large variables.
  
  \end{proof}

\subsection{Proof of Theorem~\protect{\ref{HeatThm2}}}\label{Proof1Sec}

\begin{proof}[Proof of Theorem~\ref{HeatThm2}]

We begin with~$k=0$.

We apply Duhamel's formula to~$e^{\mathbb X_t^2}$,
split the result as in~\eqref{SplitDuhamel} and use partial integration
iteratively to get rid of all terms with~$c_i=2$ for some~$i\in I$.

Thus if~$i\in I$ and~$c_i=2$,
the corresponding term~$K(t,n,I,J;\dots)$ is replaced by three or four
terms as in~\eqref{PartIntResult}.
In the first two of these terms,
the corresponding variable~$s_i$ is frozen,
whereas the remaining terms still involve an integral over~$s_i\in(0,\bar s)$,
but with~$c_i=0$.
These integrals persist if we perform more partial integrations,
so the remaining terms do not contribute to the limit as~$t\to\infty$
by Proposition~\ref{OtherTermsProp}.

We also note that whenever any term contains~$c_i\in\{0,1\}$ for some~$i\in I$
or~$c_i>0$ for some~$i\notin I$,
then this fact is not altered by partial integration,
so these terms also do not contribute in the limit
by Proposition~\ref{OtherTermsProp}.
Thus, only those terms $K(t,n,I;c_0,\dots,c_n;a_1,\dots,a_n)$
in equation~\eqref{SplitDuhamel} contribute to the limit
where
	$$c_i=\begin{cases}2&\text{if~$i\in I$, and}\\
		0&\text{if~$i\notin I$.}\end{cases}
	$$
Whenever~$c_i=2$ and~$i\in I$, the corresponding part of the integrand
in such a term
must be of the form
\begin{equation}\label{LetterA}
	\dots e^{s_{i-1}tD^2}\,R_t\,\sqrt t\,D\,e^{s_itD^2}\,\sqrt t\,D\,R_t
		\,e^{s_{i+1}tD^2}\dots\;,
\end{equation}
whence~$0<i<n$, $i-1$, $i+1\notin I$ and~$a_i=a_{i+1}=1$.
On the other hand,
if~$i-1$, $i\notin I$, the corresponding part of the integrand takes
the form
\begin{equation}\label{LetterB}
	\dots e^{s_{i-1}tD^2}\,R_t^2\,e^{s_itD^2}\dots\;,
\end{equation}
whence~$a_i=2$ in this case.
Thus,
the summands~$K(t,n,I;\dots)$ that contribute to the limit
are in one to one correspondence with finite words in the free ring generated by
the two letters~$A$ and~$B$,
where each~$A$ stands for an occurrence of~\eqref{LetterA}
and each~$B$ stands for~\eqref{LetterB}.
Two subsequent terms overlap at~$e^{s_itD^2}$ with~$i\notin I$,
and the empty word represents~$e^{tD^2}$.
Note however that the mapping from this ring to $\Omega^\bullet(B,\End_{\mathcal A}\Omega^\bullet(E/B;\mathcal F))$ that assigns to each monomial a term in the Duhamel expansion of the heat kernel is only additive, not a homomorphism.
Because each letter contains~$R_t$ twice,
its degree with respect to~$B$ is at least~$2$,
so there cannot be more than~$\frac{m_B}2$ letters.

Partial integration now has the effect of replacing one letter~$A$
by~$C-B$, where the letter~$C$ stands for
\begin{equation}\label{LetterC}
	\dots e^{s_{i-1}tD^2}\,R_t\,e^{\bar s(t)\,tD^2}
	\,R_t\,e^{s_itD^2}\dots\;,
\end{equation}
modulo terms that vanish in $\tau$-norm as~$t\to\infty$.
As in Proposition~\ref{OtherTermsProp},
the word~$C^n$ converges to
	$$\frac1{n!}\,P\,(R_0P)^{2n}\;.$$
We thus find that
\begin{equation}\label{LimitFormula}
  \lim_{t\to\infty}e^{\mathbb X_t^2}
  =\lim_{t\to\infty}\sum_{n=0}^{\left\lfloor\tfrac{m_B}2\right\rfloor}(A+B)^n
  =\lim_{t\to\infty}\sum_{n=0}^{\left\lfloor\tfrac{m_B}2\right\rfloor}C^n
  =P\,e^{(R_0P)^2}\;.
\end{equation}
This completes the proof of Theorem~\ref{HeatThm2} for~$k=0$.

The analogues of~\eqref{int-spez} and~\eqref{SplitDuhamel} for~$k=1$
are given by
\begin{equation}\label{SplitDuhamel2}
  \mathbb X_t\,e^{\mathbb X_t^2}
  =\sum_n\sum_I\sum_{
    \begin{smallmatrix}
      \phantom{a_1+c_1}a_0\le 1\le a_0+c_0\phantom{\dots}\\
      a_0+c_0+a_1\le 3\le a_0+\dots+c_1\\
      \vdots\\
      a_0+\dots+a_n\le 2n+1=a_0+\dots+c_n
    \end{smallmatrix}}
  K(t,n,I;c_0,\dots,c_n;a_0,\dots,a_n)\;,
\end{equation}
where
\begin{multline*}
K(t,n,I;c_0,\dots,c_n;a_0,\dots,a_n)
=\int_{\Delta_{\bar s(t),I}^n}R_t^{a_0}
(\sqrt t\,D)^{c_0}e^{s_0 tD^2}R_t^{a_1}(\sqrt t\,D)^{c_1}e^{s_1tD^2}\dots \\
\dots (\sqrt t\,D)^{c_n}e^{s_ntD^2}\,d^n(s_0,\dots,s_n)
\end{multline*}
with~$c_i\ge 0$ and~$a_i>0$ for all~$i$.

We perform partial integration as before.
By the analogue of Proposition~\ref{OtherTermsProp},
the remaining terms can again be described by letters~$A$, $B$ and~$C$
as above,
where we have to delete the leftmost~$e^{s t\,D^2}\,R_t$
from the first letter in each word.
Counting the number of free simplex variables correctly,
we find that
	$$\lim_{t\to\infty}C^{n+1}=\frac 1{n!}\,P(R_0P)^{2n+1}\;.$$
With these modifications,
the limit in the $\tau$-norm can now be described as
\begin{equation}\label{LimitFormula2}
  \lim_{t\to\infty}\mathbb X_t\,e^{\mathbb X_t^2}=
 \lim_{t\to\infty}\sum_{n=0}^{\left\lfloor\tfrac{m_B-1}2\right\rfloor}(A+B)^{n+1} =PR_0P\,e^{(R_0P)^2}\;.
\end{equation} 

For~$k=2$, we similarly consider the Duhamel expansion of~$\mathbb X_t\,e^{\mathbb X_t^2}\,\mathbb X_t$, leaving the details to the reader. We still work with letters~$A$, $B$, $C$ as before, where we delete both the leftmost~$e^{stD^2}$ from the first letter and the rightmost~$e^{stD^2}$ from the last letter in each word. For the limit in the $\tau$-norm, we obtain
\begin{equation}\label{LimitFormula3}
 \lim_{t\to\infty}\mathbb X_t\,e^{\mathbb X_t^2}\,\mathbb X_t
 =\lim_{t\to\infty}\sum_{n=0}^{\left\lfloor\tfrac{m_B-2}2\right\rfloor}(A+B)^{n+2}
 =PR_0P\,e^{(R_0P)^2}\,R_0P\;.
\end{equation}

\qedhere
\end{proof}

\section{\texorpdfstring{$L^2$-}{L2}index theorems}
\label{secL2indthm}
%%%%%% section B I S M U T   L O T T 

\subsection{\texorpdfstring{$L^2$-}{L2}Bismut--Lott theorem}
\label{secBL}
Our first application of Theorem \ref{HeatThm2} is  the $L^2$-Bismut--Lott index theorem. This was proved by Gong and Rothenberg in \cite{GR} assuming extra regularity hypothesis.

Let $(\tilde E, \Gamma)\to B$ be a family of normal coverings, and $M\to B$ be
a family of finitely generated Hilbertian $\Gamma$-$\mathcal A$-bimodules as
in Definition \ref{def1.2}. We use here the Euler grading. The following
theorem proves that the $L^2$-Kamber--Tondeur class of the flat  bundle of
$\cA$-modules $H_{L^2}(E/B;\cF)=\bigoplus_k(-1)^k H_{L^2}^k(E/B;\cF)\to B$ is
equal to the Becker--Gottlieb transfer of the class of $\cF$ (see definitions in  Section \ref{L2KamTon}).
\begin{theorem}\label{L2BL} If $\dim Z$ is even,
$$
\ch{}^\circ_\tau(H_{L^2}(E/B;\cF))=\int_{E/B}e(TZ)\ch{}^\circ_\tau(\cF)\;\;\; \in \, H^{odd}_{dR}(B).
$$
\end{theorem}
\begin{proof} Let $f(a)=a \exp(a^2)$, and let $F_\t(t):=\sqrt{2\pi i}\Phi
  \Str_\tau(f(\mathbb X_{t}))$ as defined in \eqref{Ftau}, with $\Phi(\a)=
  (2\pi i)^{-\frac{|\a|}{2}}\a$. $F_\t(t)$ is a closed, real odd form on $B$,
  and by \eqref{trF^}, its cohomology class does not depend on $t$. The small
  time limit of $F_\t(t)$ can be obtained as in \cite[Theorem 3.16]{BL} and
  gives, as $t\to 0$ 
\begin{equation*}
\label{tto0}
F_\t(t)=\begin{cases}{}\ds\int_{E/B} e(TZ,
    \nabla^{TZ})\ch{}^\circ_\t(\mathcal{F}, g^{\mathcal F})+\mathcal O(t)\;,\; &
    \text{if } \dim Z \text{ is even}
 \\\mathcal O(\sqrt{t})\;,\;& \text{if } \dim Z \text{ is odd}\ .\end{cases}
\end{equation*}
On the other hand, Theorem \ref{HeatThm2} implies
$$\lim_{t\to\infty}\Str_\tau\left( \mathbb {X}_t e^{\mathbb{X}_t^2}\right)=
\Strt\left( P(R_0P)e^{(R_0P)^2}\right)\ .$$ 
Since $PR_0P=\frac{1}{2}(\nabla
^{H,*}-\nabla^H)$, it follows immediately that
\begin{equation*}
\lim_{t\to\infty} F_\t(t)= \ch{}^\circ_\t(H_{L^2}(E/B;\cF), g^{H_{L^2}})\ . \qedhere
\end{equation*}
\end{proof}
We get a family version of Atiyah's $L^2$-index theorem as a special case:
\begin{cor}
In the situation of Example \ref{ex1.2}, when $\cF=F\otimes\mathcal{L} $ with
$\mathcal{L}=\tilde E\times_{\Gamma} l^2(\Gamma)$ comes from a finite
dimensional flat vector bundle  $F\to E$, then 
$$
\ch{}^\circ_\t(H_{L^2}(E/B;\cF), g^{H_{L^2}})=\ch{}^\circ(H(E/B;F), g^{H})\ .
$$
\end{cor}

\subsection{\texorpdfstring{$L^2$-}{L2}index theorem for the family of signature operators}\label{L2index}
Our next application of Theorem \ref{HeatThm2} is the $L^2$-index theorem for the families of signature operators twisted by a flat duality bundle. 

\begin{theorem}\label{L2indT}
Let $Z\to E\stackrel{p}{\to}B$ be a smooth fibre bundle with connected even-dimensional closed fibres, let $\mathcal F\to E$ be a flat bundle of $\mathcal A$-modules as in \eqref{coeff-bundle} or \eqref{coeff-bundle2} with a flat duality structure.
Then
\begin{equation} 
\label{ind-formula}
p_\t (\nabla^H, J^H)=\int_{E/B}L(E/B)\,p_\t(\nabla^{\cF}, J^\cF)\;\;\;\in  H_{dR}^*(B)\ .
\end{equation}
\end{theorem}
\begin{proof}
By \eqref{McKt}, the cohomology class of $\tr{}_\t(J e^{-\mathbb{A}_t^2})$ is constant with respect to $t$. 
The small time limit of $\tr{}_\t(J e^{-\mathbb{A}_t^2})$ is computed as in \cite[Proposition 31]{Lo2}, \cite[3.16]{BL} and \cite{Bi} 
and gives 
$$
\lim_{t\rightarrow 0}\tr{}_\t(J e^{-\mathbb{A}_t^2})=\int_{E/B}L(E/B)\, p_\t(\nabla^{\cF}, J^\cF)\ .
$$
The large time limit is provided by Theorem \ref{HeatThm2}: 
$$\lim_{t\rightarrow \infty}\trt(J e^{-\mathbb{A}_t^2})=\trt(J Pe^{(R_0P)^2})\ .
$$ 
%By \eqref{omegaH},  
%$PR_0P=\frac{1}{2}P(\nabla^{\cW,*}-\nabla^\cW)P=\frac{1}{2}((\nabla^H)^*-\nabla^H)=\frac{1}{2}(J[\nabla^H,J])$. 
Comparing with the computation of Lemma \ref{comput.p}, we then have
 \begin{equation}
\lim_{t\rightarrow \infty}\trt(J e^{\mathbb{X}^2_t})=p_\t (\nabla^H, J^H)
\end{equation}
and the equality \eqref{ind-formula} follows directly from the McKean--Singer formula \eqref{McKt}.
\end{proof}

Again a special case is a
version of Atiyah's $L^2$-index theorem for families of twisted signature
operators.
\begin{cor}
Consider the situation of Example \ref{ex1.2}, when $\cF$ comes from two
finite dimensional flat vector bundles  $F^+\oplus F^-\to E$, i.e.~$\nabla^\cF
J^\cF=0$. Then  $D^{sign}=d^Z+d^{Z, *}$ is the twisted signature operator. Denoting by $\underline{D}^{sign}$ the signature operator twisted by $F^+\oplus F^-$, then in $H_{dR}^*(B)$ we have
$$
\ch_\t \Ker D^{sign}=\ch \Ker \underline{D}^{sign}\ .
$$

\end{cor}
\subsection{Remarks}
\subsubsection{Index class versus index bundle}\label{Ib-Ic}

Consider the case of normal coverings of fibre bundles. From the point of view of non-commutative geometry, the family of twisted signature operators $D^{sign}$ possesses an analytic index class $ \aI D^{sign}\in K_0(C(B)\otimes C^*(\G))$. More generally, the index class belongs to the $K$-theory of a certain groupoid $\aI D^{sign}\in K_0(C_c^\infty(\mathcal G))$.
This class represents the obstruction to invertibility in $C_c^{-\infty}(\mathcal G)$ of the operator $ D^{sign}$ which is invertible modulo $C^\infty _c(\mathcal G)$ (\cite[II.9.$\alpha$]{Co}). 

In the classical case of a compact fibre family, the index class of the family of operators coincides with the $K$-theory class of the index bundle for any family of Dirac operators whose kernels form a bundle.  This is no longer true on non-compact fibres/leaves, where, basically, the obstruction to invertibility needs not be \lq\lq concentrated in the kernel bundle".

The question of the equality of the index class and the index bundle once one
has paired the Chern character with a trace, was first investigated by
Heitsch, Lazarov and Benameur in the more general situation of a foliated
manifold with Hausdorff graph. The results in \cite{HL, BH4} guarantee it is
true if the spectrum of $D$ is very well behaved (smoothness of the spectral
projection relative to $(0,\ep)$ plus a lower bound on the Novikov--Shubin
invariants). An example where the equality fails is given by Benameur, 
Heitsch and Wahl on a Lusztig fibration in \cite{BHW}.

Our Theorem \ref{L2indT} proves the desired equality for the signature operator with coefficients in a globally flat bundle, in the setting of families of normal coverings. 

\begin{cor} \label{pair}
Let $(\tilde E, \Gamma)\to B$ a family of normal coverings, and $M\to B$ be a family of flat finitely-generated Hilbertian $\Gamma$-$\mathcal A$-bimodules as in Definition \ref{def1.2}.
In this situation the pairing of the index bundle and of the index class with elements in $H_*(B)\otimes \tau$ are equal.
\end{cor}

\subsubsection{Lusztig fibrewise flat twisting bundle}\label{Lusztig}

Our methods do not extend to the fibrewise flat case because the operator $d^E$ is no longer a flat superconnection and the property $(d^{E,*}+d^E)^2=-(d^{E,*}-d^E)^2$ is no longer true. 
This is consistent with \cite{BHW}.

\subsubsection{Examples of spectral density}

One can construct an example of a badly behaved spectral density function, starting from the Lusztig fibration.
 
Consider $\pi_2\colon T\times T^*\rightarrow T^*$ where $T=S^1=[0,2\pi]/\sim$, and $T^*$ is the dual torus which we parametrize as $T^*=\{\theta_s\colon\mathbb Z\rightarrow U(1), n\mapsto e^{2\pi i n s}\}$, $s\in [0,1]/\sim$.
The line bundle $l\colon L\rightarrow T\times T^*$ defined as $L=(\mathbb R\times \mathbb R^*\times \mathbb C)/(\mathbb Z\times \mathbb Z^*)$ with the action $(n,m)\cdot (t,r,\lambda)=(2\pi n+t, m+r, e^{2\pi i\langle r,n\rangle}\la)$ is leafwise flat, because $L_\theta=L_{|\pi^{-1}(\theta)}$ is flat. 
Let $(D_{s})_{s\in T^*}$ be the family of signature operators twisted by $L_{\theta_s}$: it is explicitly given by $D_s=i\partial_t$ on $\mathcal C^\infty(S^1,L_{\theta_s})=\{g:[0,2\pi]\rightarrow \mathbb C\,\,\,| \,\; e^{2\pi is}g(0)=g(2\pi)\}$.
We have $\spec D_s=\{(k+s), k\in \mathbb Z\}$, so that  the spectral density function of $D_s$ is equal to $F_s(\la)=\tr E_{\la^2}^{\Delta_s}= \lfloor \la+1-s \rfloor+\lfloor\la+s\rfloor$. 
Let $X$ be a closed manifold, whose universal covering $\tilde X$ is such that the Laplacian on $\tilde X$ has a nontrivial kernel, and let $\pi_3\colon T\times X\times T^*\to T^*$ be the fibration having as fibre the product manifold $T\times X$.

Consider now the family of normal $\pi_1(X)$-covering $q_3\colon T\times \tilde X\times T^*\to T\times X\times T^*$. Lift the twisting bundle $l$ to the product $T\times X\times T^*$ and to the covering. Computing the spectral density function of the Dirac operator on the product  (using the convolution of the densities) one can see that it has a discontinuity in  
 $s=0$.

%%%%%   section      E T A
\section{Refined index theorems and secondary invariants}
\label{refined}
 In this section we prove the refinements of Theorems \ref{L2indT} and
 \ref{L2BL} at the level of differential forms, and we define the $L^2$-eta
 form and the $L^2$-higher analytic torsion. To this aim, we look for the
 weakest regularity condition under which we can pass to the large time limit
 in the transgression formulas derived from  \eqref{trF^} and
 \eqref{McKt}. 
Making  use of the estimates of Section \ref{heat-kernel}, we show that the secondary invariants eta and torsion are well defined in the following two cases: if the typical fibre has positive Novikov--Shubin invariants, or if it is of determinant class and $L^2$-acyclic.

The $L^2$-torsion form was first introduced by Gong and Rothenberg \cite{GR},
assuming  much stronger regularity hypothesis (smoothness of the spectral
projection $\chi_{(0,\ep)}(D)$ and positive Novikov--Shubin invariants). Our
extension to certain families of determinant class is relevant, because it was
recently proved by Grabowski that there exist closed manifolds
with Novikov--Shubin invariant equal to zero \cite{Gra}, but these examples are
of determinant class by \cite{Sc2}.

\subsection{\texorpdfstring{$L^2$}{L2}-torsion forms}
Consider the $L^2$-Betti numbers with coefficients in $\cF$ defined by $b_\tau^{(k)}(Z, \cF)=\dim_\t\left( \ker(d^Z+d^{Z,*})\cap \cW^k \right)$, and define the \emph{$L^2$-Euler characteristic}, and the \emph{derived $L^2$-Euler characteristic}, respectively as
$$\chi_\t(E/B):=\sum_k (-1)^kb_\tau^{(k)}(Z, \cF)\;\;,\;\;\;\;\; \;\chi_\t'(E/B):=\sum_k (-1)^k k b_\tau^{(k)}(Z, \cF)$$   
\begin{lem}
\label{limFhat}
$$
\lim_{t\to\infty} F_\tau^{\wedge}(t)=\frac{\chi'_\t (E/B)}{2} \ .
$$
\end{lem}
\begin{proof}
It is enough to apply Theorem \ref{HeatThm2}.  Because $\Str_\t\left(NP(R_0P)^{2j}\right)=0 \;\;\forall j\neq 0$, it follows that
$\lim_{t\to\infty} \Str_\t \left(\frac{N}{2}(1+2\mathbb X_t^2) e^{\mathbb X_t^2}\right)=\Str_\t \left(\frac{N}{2}\right)$. 
\end{proof}

\begin{lem}[Theorem 3.20 in \cite{BL}]
As $t\to 0$, 
$$
F^{\wedge}_\tau(t)=\left\{\begin{array}{cc}\frac{1}{4}\dim Z \rk_\tau(\mathcal
    F)\chi_\t(E/B) , & \dim Z \text{ even} \\
\mathcal O(\sqrt t) , &\dim Z \text{ odd}\ .
\end{array}\right.
$$
\end{lem}

The integral of $\frac{ F^{\wedge}_\t(t)}{t}$ would diverge both for $t\to 0$ and $t\to\infty$, so we add the usual compensation scalar terms as in \cite[Def 3.22]{BL} and define the function
\begin{equation}
\label{TI}
\mathcal{T}_\t(T^HE, g^{TZ}, g^{\mathcal F})(t):= 
\frac{1}{t} \left[F_\t^{\wedge}(t)-\frac{\chi'_\tau}{2}-\left(\frac{n \rk_\tau \mathcal F\cdot \chi_\tau}{4} -\frac{\chi_\tau'}{2}\right)(1-2t)e^{-t}\right] \ .
\end{equation}

\begin{dfn} \label{defdetcl}
Denote as before $D=\frac{1}{2}(d^{Z,*}-d^Z)$, and recall that $P$ is the projection onto $\Ker D$.
The fibre $Z$ is called of \emph{determinant class} if 
\begin{equation}
 \int_1^\infty\trt (e^{tD^2}-P)\frac{dt}{t}<\infty\ .
 \end{equation}
\end{dfn}
The first statement of the following proposition was proved by Gong and Rothenberg in \cite{GR} under the additional hypothesis that the  spectral projections $P_\ep$ are smooth. 

\begin{prop} 
\label{Htor} 
\begin{enumerate}
\item[(I)] %If there exists $a>0$ such that $\;\theta(t)=\mathcal O(t^{-a})$ as $t\to \infty$, 
If the Novikov--Shubin invariants are positive, then there exists $\ep>0$ such that in $\Omega^\bullet(B; \End^1_{\cA}\Omega_{L^2}(E/B))$, \emph{i.e.} in the trace norm
$$
F_\t^{\wedge}(t)-\frac{\chi'_\tau}{2}= \mathcal O(t^{-\ep}) \;\;, \;\; \text{as } t\to \infty\ .
$$
\item[(II)] If $Z$ is of determinant class and $L^2$-acyclic, then 
$$
\ds \int_1^\infty \frac{1}{t} F_\t^{\wedge}(t)dt< \infty\ .
$$
\end{enumerate}
\end{prop}

\begin{proof}
We go back to the proof of Theorem \ref{HeatThm2}. 
Consider the expression for $F_\t^{\wedge}(t)$ developed with the Duhamel expansion as in Section \ref{Proof1Sec}. By Lemma \ref{limFhat}, it is enough to estimate all the terms in the expansion that go to zero as $t\to\infty$. 
 
If the hypothesis of (I) is verified, with the choices of Lemma \ref{barst} (2), there exists $\ep>0$ such that $\theta(t\bar s(t))\cdot \left(\frac{1}{\bar s(t)}\right)^{m_B}\leq t^{-\ep}$, as $t\to\infty$. Therefore the remainder terms in \eqref{rest1} and \eqref{rest3}, summarized in \eqref{esti}, become an $\mathcal O(t^{-\ep})$.

If we have the hypothesis of part (II), then we choose $\bar s(t)$ as in Lemma \ref{barst} (3) and we prove that all remainder terms are integrable on $[1,\infty)$.
To do so, we proceed as in the proof of Proposition \ref{EasyTermsProp} and
\ref{OtherTermsProp}. In particular, we successively replace $e^{\bar s t
  D^2}=(e^{\bar s t D^2}-P)+P$  and apply the determinant class condition to
$e^{\bar s t D^2}-P$. Since we are assuming $L^2$-acyclicity,  $P=0$ hence the
remainder terms with $P$ are not there, and the convergence holds.
\end{proof}

\begin{cor}
\label{torcor}
If any of the two hypothesis in (I) or (II) of Proposition \ref{Htor} is verified, then the $L^2$-torsion form is well defined as
\begin{equation}
\label{torform}
\mathcal T_\tau(T^HE, g^{TZ},g^\cF)=-\int_0^\infty  \left[F_\t^{\wedge}(t)-\frac{\chi'_\tau}{2}-\left(\frac{n\rk_\tau \mathcal F\cdot \chi_\tau}{4} -\frac{\chi_\tau'}{2}\right)(1-2t)e^{-t}\right]\frac{1}{t}dt \ .
\end{equation}
\end{cor}

\medskip

 \begin{rem} 
\emph{$L^2$-torsion and Igusa's axioms: a question}.
The higher analytic torsion of Bismut and Lott has a counterpart in (differential) topology, the \emph{higher Franz--Reidemeister torsion}  $\mathcal T_{IK}$ defined by Igusa and Klein \cite{Ig}.

Igusa gave a set of axioms characterizing $\mathcal T_{IK}$ in the case of a smooth unipotent fibre bundle $p\colon E\rightarrow B$, and without coefficients \cite{I}.
It is an open question how to axiomatize higher torsion when the input data also contains a flat twisting bundle, \emph{i.e.} a representation $\phi\colon\pi_1(E)\rightarrow U(n)$.
As explained in \cite{I1}, the set of desired axioms should contain an additional \lq\lq continuity condition" with respect to the representation $\phi$. 

We think that the axiom could be a continuity condition on the sequence of
higher analytic torsions for a tower of coverings, so involving the
$L^2$-torsion defined in \ref{torform}. More precisely, we ask whether, given a
residually finite covering of a fibre bundle (possibly under opportune
assumptions), the sequence of Bismut--Lott torsions for the finite covering
families converges to the $L^2$-higher torsion.
Such a property, if true, could provide an interesting basis for future investigation of $L^2$-topological higher torsion invariants.
\end{rem}

\subsection{\texorpdfstring{$L^2$}{L2}-eta forms for the signature operator}

Let $Z\to E\stackrel{p}{\to}B$ be a smooth fibre bundle with connected
$2l$-dimensional closed oriented Riemannian fibres, let $\mathcal F\to E$ be a flat bundle of $\mathcal A$-modules as in \eqref{coeff-bundle} or \eqref{coeff-bundle2}.

If $\dim Z=2l$ consider in $\O(B,\cW)^{dua}$, recall \eqref{McKt},
 the eta function
\begin{equation}\label{etaevenint}
\eta_\t(t):= (2\pi i)^{-\frac{1}{2}}\Phi \tr_\t \left(J \frac{d\mathbb{A}_t}{dt}e^{-\mathbb A^2_t}\right)\ .
\end{equation}

If $\dim Z:=2l+1$, we consider the \emph{odd signature operator}  of \ref{tra-odd}
and we set $$\eta_\tau(t):=(2\pi i)^{-\frac{1}{2}}\Phi \trt \left(\frac{d\mathbb{A}_t}{dt}e^{-\mathbb{A}_t^2}\right)^{even}\ .$$
\begin{lem}
In both even and odd dimensional cases
\begin{equation}
\lim_{t\to\infty}\eta_\tau(t)=0\ . 
\end{equation}
\end{lem}
\begin{proof}
Consider first even dimensional fibres. By Remark \ref{dterm.rem}, we look at the $d\vartheta$-term of $\tr \left(J e^{\breve{\mathbb X}^2}\right)=\Str\left((\hat \o_\C  \otimes J^\cF )e^{\breve{\mathbb X}^2}\right)$.  We compute its large time limit with Theorem \ref{HeatThm2}:
\begin{equation}
\lim_{t\to\infty} \Str\left((\hat \o_\C  \otimes J^\cF )e^{\breve{\mathbb X}^2}\right)= \Str\left((\hat \o_\C \otimes J^\cF) Pe^{(\breve{R}_0P)^2}\right)
\end{equation}
%$\sum_{k\geq 0}\Str \frac{1}{k!}\left(\hat \o_\C P(\hat R_0 P)^{2k}\right)=\sum_{k\geq 0}\Str\frac{1}{k!}\left(\hat \o_\C P(R_0+\left(N-\frac{n}{2}\right)d\vartheta)^{2k}\right)$
%\end{multline*}
where $\breve{R}_0=R_0+\left(N-\frac{n}{2}\right)d\vartheta$.
Then the $d\vartheta$-term of the right hand side is equal to
$$\sum_{k\geq 0} \frac{1}{k!}\Str\left((\hat \o_\C  \otimes J^\cF) P(R_0 P)^{2k+1}(N-\frac{n}{2})P\right)
$$
 which is equal to zero because for each $k$
\begin{multline*}
  \Str\left((\hat \o_\C  \otimes J^\cF) P(R_0 P)^{2k+1}(N-\frac{n}{2})P\right)=\Str\left((\hat \o_\C  \otimes J^\cF)P(R_0 P)^{2k}(R_0P)(N-\frac{n}{2})P\right)\\
  = \Str\left((\hat \o_\C  \otimes J^\cF) P(R_0 P)^{2k}(N-\frac{n}{2})(R_0P)\right)= \Str\left((N-\frac{n}{2})(R_0P)(\hat \o_\C  \otimes J^\cF)  P(R_0 P)^{2k}\right)\\
  =-\Str\left((N-\frac{n}{2})(\hat \o_\C  \otimes J^\cF) P(R_0 P)^{2k+1}\right)=-\Str\left( (\hat \o_\C  \otimes J^\cF) P(R_0 P)^{2k+1}(N-\frac{n}{2})P\right)
\end{multline*}
where we have used that $R_0P$ anti-commutes with $\hat \o_\C  \otimes J^\cF$, and that $N-\frac{n}{2}$ is even.
 
 For odd dimensional fibres, the computation is similar.
\end{proof}

\begin{prop} 
\label{Heta}  
\begin{itemize}
\item[(I)] If the Novikov--Shubin invariants are positive, then there exists $\ep>0$ such that in $\Omega^\bullet(B; \End^1_{\cA}\Omega_{L^2}(E/B))$, \emph{i.e.} in the trace norm
$$
\eta_\tau(t)= \mathcal O(t^{-1-\ep}) \;\;, \;\; \text{as } t\to \infty \ .
$$

\item[(II)] If $Z$ is of determinant class and $L^2$-acyclic, then $\ds \int_1^\infty \eta_\tau (t)dt< \infty$ .
\end{itemize}
\end{prop}

\begin{proof}
 Consider the Duhamel expansion of $e^{\breve{\mathbb X}_t^2}$, split it as in \ref{spl-Duham} and integrate by parts. Then we proceed exactly as in the proof of Proposition \ref{Htor}, choosing $\bar s(t)$ as in Lemma~\ref{barst} part (2) or as in part (3), respectively, in the two cases. 

%We substitute $e^{\bar s t D^2}=(e^{\bar s t D^2}-P)+P$  and apply the determinant class condition to $e^{\bar s t D^2}-P$, obtaining that all such remainder terms are integrable. Assuming $L^2$-acyclicity, $P=0$, hence the terms with $P$ are not there, and the convergence holds.
\end{proof}

\begin{cor}\label{etacor}
If any of the two hypothesis in (I) or (II) of Proposition \ref{Heta} is
verified, then the $L^2$-eta form of the signature operator is well defined as
\begin{equation}
\label{etaform}
\eta_\tau(\mathcal{D}^{sign})=
\begin{cases}
\ds(2\pi i)^{-\frac{1}{2}}\Phi\int_0^\infty\trt\left(\frac{d\mathbb{A}_t}{dt}e^{-\mathbb{A}_t^2}\right)^{even}dt\,,  &\dim Z=2l+1\\
\ds(2\pi i)^{-\frac{1}{2}}\Phi\int_0^\infty\trt\left(J\frac{d\mathbb{A}_t}{dt}e^{-\mathbb{A}_t^2}\right)dt\,,&\dim Z=2l\ . 
\end{cases}
\end{equation}
\end{cor}
  
If the fibre is odd-dimensional, then the zero-degree part of $\eta_\t(D^{sign}_{odd})$ is a function on $B$ whose value at the point $b$ is equal to the Cheeger--Gromov $L^2$-eta invariant of $Z_b$, introduced in \cite{CG}.
Guided by the fact that Cheeger and Gromov prove the existence of the
$L^2$-eta invariant of the signature operator without any condition, we
ask the following. 
 \begin{que}
Can the $L^2$-eta form of the signature operator be defined dropping the
$L^2$-acyclicity condition and all other extra conditions?
 \end{que}

\subsection{Local index theorems}

The proofs of Propositions \ref{Htor} and \ref{Heta} defined $\cT_\t$ and
$\eta_\tau$ as continuous differential forms on $B$. Our estimates are not
good enough to prove that $\eta_\tau$ is a $C^1$ differential form. Nevertheless,
we can use weak exterior derivatives to prove  local index theorems. Gong and
Rothenberg  proved the same result under stronger regularity hypothesis
\cite[Th. 3.2]{GR}.
 \begin{dfn} 
 A continuous $k$-form $\phi$ on $B$ is said to have weak exterior derivative $\psi$ if  for every smooth $(k+1)$-simplex $c\colon \Delta^{k+1}\rightarrow B$ 
  $$
 \int_c\psi=\int_{\partial c}\phi \ .
 $$
\end{dfn}

Let $Z\to E\stackrel{p}{\to}B$ be a smooth fibre bundle with connected closed fibres, let $\mathcal F\to E$ be a flat bundle of $\mathcal A$-modules as in \eqref{coeff-bundle} or \eqref{coeff-bundle2}.%with possibly a flat duality structure.
\begin{theorem}
\label{weak-thm}
Assume the fibres $Z$ have positive Novikov--Shubin invariants. Then the 
 form $\mathcal{T}_\tau$ has weak exterior derivative
$$
d\mathcal{T}_\tau=\int_{E/B}e(TZ, \nabla^{TZ})\ch{}^\circ(\cF, g^\cF)- \ch{}^\circ_\t(H_{L^2}(E/B;\cF), g^{H_{L^2}})\ .
$$
If the fibres are determinant class and $L^2$-acyclic, then $\mathcal{T}_\tau$ has weak exterior derivative
$$
d\mathcal{T}_\tau=\int_{E/B}e(TZ, \nabla^{TZ})\ch{}^\circ(\cF, g^\cF)\ .
$$
\end{theorem}
\begin{proof}
Let $c$ be a $(k+1)$-smooth chain in $B$. By \eqref{trF^} and the theorem of Stokes, on a finite interval $0<t<T<\infty$ we have
$$
\int_c F_\tau(t)-\int_c F_\tau(T)=-\int_{\partial c}\int_t^T \frac{1}{t}\left(F_\t^{\wedge}(t)-\frac{\chi'_\tau}{2}-\left(\frac{n \rk_\tau \mathcal F\cdot \chi_\tau}{4} -\frac{\chi_\tau'}{2}\right)(1-2t)e^{-t}\right)dt
$$
which implies the desired result.
\end{proof}
Let now $Z\to E\stackrel{p}{\to}B$, $\mathcal F\to E$ be  as above and let $\cF$ be endowed with a flat duality structure. The following local formulas are deduced in the same way from \eqref{McKt} and \eqref{trasgr-odd}.

\begin{theorem}
\label{weak-thm2}
Assume $\dim Z=n=2k$. If the fibres have positive Novikov--Shubin invariants, then
the form $\eta_\tau$ has weak exterior derivative
$$
d\eta_\tau=\int_{E/B} L(E/B)\;p_\tau (\nabla^\mathcal{F}, J^\cF) -p_\tau(\nabla^{H_{L^2}}, J^{H_{L^2}})\ .
$$
If the fibres are of determinant class and $L^2$-acyclic, then the same holds, with the last term on the right hand side vanishing.

If $n=2k+1$, and assuming either positive Novikov--Shubin, or determinant class and $L^2$-acyclicity, then the form $\eta_\tau$ has weak exterior derivative
$$
d\eta_\tau=\int_{E/B} L(E/B)\;p_\tau (\nabla^\mathcal{F}, J^\cF) \ .
$$
\end{theorem}

\section{\texorpdfstring{$L^2$}{L2}-rho form}
  Let $\pi \colon \tilde E\to E$ be a normal $\G$-covering of the fibre bundle
  $p\colon E\to B$ as in Section \ref{Sect.NCF}. Recall that in this case
  $\mathcal A=\mathcal N\G$, $M=l^2\G$, and  $\t$ is the canonical trace on
  $\mathcal N\G$.
  
 To define the $L^2$-rho form of the family of signature operators in this setting, we introduce the following notation: let $\underline{D}^{sign}$ be the family of signature operators along the compact fibres of $E\rightarrow B$ (\emph{i.e.} untwisted), and $D^{sign}$ be the family of signature operators twisted by $\mathcal F=\tilde E\times_\G l^2\G$. 
\begin{dfn}
If $Z$ is of determinant class and $L^2$-acyclic, or if $Z$ has positive Novikov--Shubin invariants, the \emph{$L^2$-rho form} is the difference
$$
\rho_\t(D^{sign}):=\eta_\t(D^{sign})-\eta(\underline{D}^{sign})\in C^0(B, \Lambda^* B) \ .
$$
\end{dfn}
If the fibre is odd-dimensional, then the zero-degree part of $\rho_\t(D^{sign}_{odd})$ is a function on $B$ whose value at the point $b$ is equal to the Cheeger--Gromov $L^2$-rho invariant of $Z_b$, \cite{CG}.

The local index theorem implies the following.

\begin{lem}
\label{wCL}
The $L^2$-rho form $\rho_{\tau}(D^{sign})$ is weakly closed in the following cases:\begin{itemize}
\item for odd dimensional fibres, whenever it is well defined (conditions of Corollary \ref{etacor});
\item  for even dimensional fibres, when $Z$ is of determinant class, acyclic and $L^2$-acyclic.
\end{itemize}

\end{lem}
\begin{proof} 
It suffices to look at the weak local index formulas to get the desired equality.
\end{proof}

The following proposition shows that, as usual, when the form
$\rho_{\tau}(D^{sign})$ is closed, then its cohomology class is independent of
the vertical metric $g^{TZ}$. This is in analogy with \cite[Theorem 3.24 and
Corollary  3.25]{BL}.

\begin{prop} Let $(g_u)_{u\in[0,1]}$ be a path of metrics on the vertical tangent bundle $T(E/B)$, and $D_u^{sign}$ the corresponding family of signature operators. Let us assume that $\rho_{\tau}(D_u)$ is weakly closed. 
Suppose $Z$ is determinant class and $L^2$-acyclic. Then the cohomology class of $\rho_{\tau}(D_u)$ is constant in $u$.
\end{prop}
\begin{proof}
Let $\hat E=E\times [0,1]\rightarrow B\times [0,1]=\hat B$ the family with one added parameter $u\in [0, 1]$. 
Let $\Delta\subset B$ be a $ (k+1)$-simplex. Let $C=\Delta\times [0,1]$.
We have
\begin{equation*}
0=\int_{\partial C}\hat\rho=\int_\Delta \rho_{\tau,g_0}-\int_\Delta \rho_{\tau,g_1}+\int_{[0,1]}\int_{\partial \Delta} \hat{\rho} \ .
\end{equation*}
Because $\hat \rho$ is closed, we get  
$
\rho_{\tau\,|g_0}=\rho_{\tau\,|g_1} \in H^*(B)$.
\end{proof}
Rho-invariants have natural stability properties. For example, the
Cheeger--Gromov $\rho$-invariant of the signature operator is independent of
the metric \cite{CG2}. Indrava Roy proved the analogous stability property for
the foliated $\rho$-invariant of the longitudinal signature operator if one
has a holonomy invariant transverse measure \cite[Theorem 4.3.1]{R}.

\begin{rem}
Chang and Weinberger use the $L^2$-rho invariant of the signature operator to prove that, whenever the fundamental group contains torsion, a given homotopy type of closed oriented manifolds contains infinitely many different diffeomorphism types (using surgery theory for the construction of the manifolds) \cite{CW}.
 We could then conjecture that a similar result holds for a fiber homotopy
 type, \emph{i.e.} that one can construct and then use rho-forms to
 distinguish non-fiber diffeomorphic but fiber homotopy equivalent
 maps. Because of the stability results one can expect for this kind of
 ``degree $0$ rho-invariants'' under the assumption that the maximal
 Baum-Connes assembly map for the classifying space for free actions is an
 isomorphism \cite{PS}, one can expect that such examples will require a group
 where such a very strong isomorphism result does not hold.
\end{rem}

\begin{rem} In the situation of Example \ref{ex1.2}, one could also consider a rho-type invariant $\cT_\t(T^HE, g^{TZ}, g^\cF)-\cT(T^HE, g^{TZ}, g^F)$ for acyclic $F$ and  $L^2$-acyclic $\mathcal F$. The significance of this weakly closed form is not yet understood.
\end{rem}

\bigskip
\no{\bf Acknowledgments.} 
We wish to thank Moulay Benameur, Jean-Michel Bismut, Ulrich Bunke, James
Heitsch, John Lott, Paolo Piazza and Michael Schulze for their interest in the project.
Sara Azzali was supported by an INdAM-Cofund fellowship. She wishes
to thank also for the support by the German Research Foundation (DFG) through the
Institutional Strategy of the University of G\"ottingen during large part of
the work on this paper. Sebastian Goette was supported in parts by the DFG
special programme ``Global Differential Geometry'' and by the DFG-SFB TR
``Geometric Partial Differential Equations''. Thomas Schick was partially
funded by the Courant 
Research Center ``Higher order structures in Mathematics'' within the German initiative of excellence.

 \end{document}